\documentclass{amsart}
\usepackage{verbatim}
\usepackage{amssymb}
\usepackage{amsbsy}
\usepackage{amscd}
\usepackage{amsmath}
\usepackage{amsthm}
\usepackage{amsxtra}
\usepackage{latexsym}
\usepackage[mathscr]{eucal}
\usepackage{upgreek}
\usepackage{wasysym}

\newcommand{\finmi}{\mathfrak{m}_i}
\newcommand{\finli}{\mathfrak{l}_i}
\newcommand{\finpi}{\mathfrak{p}_i}
\newcommand{\finr}{\mf{r}}
\newcommand{\sDelta}{\Delta}
\newcommand{\UU}{\widetilde{U}}

\newcommand{\sIrr}[1]{L_{#1}^{\fing}}
\newcommand{\HC}{\Upsilon}
\newcommand{\zeroset}{\mathcal{V}}
\newcommand{\Vs}[1]{L(#1 \Lam_0)}

\newcommand{\affrho}{\widehat{\rho}}

\newcommand{\bw}[1]{\bigwedge\nolimits^{#1}}

\newcommand{\wh}{\widehat}

\newcommand{\affm}{\wh{\mf{m}}}
\newcommand{\mc}{\mathcal}
\newcommand{\mf}{\mathfrak}

\newcommand{\on}{\operatorname}
\newcommand{\affP}{\affh^*_{k,\Z}}

\newcommand{\Vg}[1]{V^{#1}(\fing)}

\newcommand{\afft}{\widehat{\mathfrak{t}}}

\newcommand{\finb}{\mathfrak{b}}
\newcommand{\finn}{\mathfrak{n}}

\newcommand{\isomap}{{\;\stackrel{_\sim}{\to}\;}}
\newcommand{\finW}{W}

\newcommand{\nc}{\newcommand}
\nc{\Hp}[1]{H^{#1}}

\newcommand{\sPi}{\Pi}

\newcommand{\affn}{\widehat{\mathfrak{n}}}
\newcommand{\affb}{\widehat{\mathfrak{b}}}
\newcommand{\affl}{\widehat{\mathfrak{l}}}
\newcommand{\sW}{W}
\newcommand{\eW}{\widetilde{W}}

\newcommand{\affh}{\widehat{\mathfrak{h}}}

\newcommand{\affg}{\widehat{\mathfrak{g}}}

\newcommand{\fing}{\mathfrak{g}}

\newcommand{\finh}{\mathfrak{h}}
\newcommand{\finm}{\mathfrak{m}}

\newcommand{\sL}{{L}}
\newcommand{\slam}{\bar{\lam}}

\newcommand{\sroots}{\roots}

\newcommand{\Lamsemi}[1]{\bigwedge\nolimits^{\frac{\infty}{2}+#1}}

\newcommand{\sBGG}{\BGG^{\fing}}
\newcommand{\Irr}[1]{L(#1)}

\newcommand{\BGG}{{\mathcal O}}

\newcommand{\N}{\mathbb{N}}
\newcommand{\Q}{\mathbb{Q}}
\newcommand{\1}{{\mathbf{1}}}

\newcommand{\dual}[1]{{#1}^*}
\newcommand{\bra}{{\langle}}
\newcommand{\ket}{{\rangle}}

\newcommand{\roots}{\Delta}

\newcommand{\Lam}{\Lambda}

\newcommand{\lam}{\lambda}
\newcommand{\ra}{\rightarrow}
\newcommand{\+}{\mathop{\oplus}}
\newcommand{\Z}{\mathbb{Z}}

\newcommand{\cprime}{$'$}
\newcommand{\inv}{^{-1}}

\renewcommand{\*}{{\otimes}}
\newcommand{\C}{\mathbb{C}}
\newcommand{\che}{^{\vee}}

\newcommand{\finp}{{\mathfrak{p}}}
\newcommand{\finl}{\mathfrak{l}}
\theoremstyle{plain}
\newtheorem{Th}{Theorem}[section]
\newtheorem*{MainTh}{Main Theorem}
\newtheorem{Pro}[Th]{Proposition}

\newtheorem{Lem}[Th]{Lemma}

\newtheorem{Co}[Th]{Corollary}

\theoremstyle{definition}

\theoremstyle{remark}

\newtheorem{Rem}[Th]{Remark}

\newcommand{\affW}{\widehat{W}}

\newcommand{\semiinf}{\frac{\infty}{2}}
\newcommand{\Zhu}{A}

\DeclareMathOperator{\Aut}{Aut}

\DeclareMathOperator{\id}{id}

\DeclareMathOperator{\End}{End}
\DeclareMathOperator{\gr}{gr}

\DeclareMathOperator{\ad}{ad}

\DeclareMathOperator{\haru}{span}

\title{Rationality of
admissible affine Vertex algebras
in the category $\BGG$
}
\author{Tomoyuki Arakawa}
\address{Research Institute for Mathematical Sciences, Kyoto University,
 Kyoto 606-8502 JAPAN}

\email{arakawa@kurims.kyoto-u.ac.jp}

\thanks{This work is partially  supported 
by the JSPS Grant-in-Aid  for Scientific Research (B)
No.\ 20340007 and JSPS Grant-in-Aid for Challenging Exploratory
Research No. 23654006}


\begin{document}
\maketitle

\begin{abstract}
We study the vertex algebras associated with
modular invariant
 representations
of affine Kac-Moody algebras
at fractional  levels,
whose 
simple highest weight modules
are classified by Joseph's characteristic varieties.
We show that an irreducible highest weight representation
of a non-twisted  affine Kac-Moody algebra
at  an admissible  level $k$
is a module over  the
associated 
simple  affine vertex algebra
if and only if it is an
 admissible representation whose  integral root system
is isomorphic to that of 
 the vertex algebra itself.
This  in particular
 proves the conjecture of Adamovi\'c and Milas \cite{AdaMil95}
on the rationality of admissible affine vertex algebras
in the category $\BGG$.
\end{abstract}

\section{Introduction}
Let $\fing$ be a complex simple Lie algebra,
$\affg$ the non-twisted affine Kac-Moody algebra associated with
$\fing$,
$\Vg{k}$   the universal affine vertex algebra
associated with $\fing$ at a non-critical level $k$,
$\Vs{k}$ the unique simple quotient of $\Vg{k}$.
The simple affine vertex algebra
$\Vs{k}$ 
is called {\em admissible} if it is isomorphic
to  an admissible representation \cite{KacWak89}
as  a $\affg$-module.
The purpose of this article is to
classify  simple
modules over 
admissible affine vertex algebras.

By a well-known result of Zhu \cite{Zhu96},
there is a one-to-one correspondence
between
 positively graded  
simple modules over
a graded  vertex algebra $V$
and simple $\Zhu(V)$-modules, where
$\Zhu(V)$ is {\em Zhu's algebra} of $V$.
In the case that
$V$ is an affine vertex algebra
$\Vs{k}$,
we have
\begin{align}
 \Zhu(\Vs{k})\cong U(\fing)/I_k
\label{eq:zhu-of-simple}
\end{align}
for some
two sided-ideal $I_k$ of the universal enveloping algebra
$U(\fing)$ of $\fing$.
Since  a simple $U(\fing)$-module $M$
is an $\Zhu(\Vs{k})$-module if and only if  the annihilating ideal of $M$
in $U(\fing)$
contains  $I_k$,
our problem amounts to 
classify the primitive ideals of $U(\fing)$ containing  $I_k$.
Because any primitive ideal of $U(\fing)$ is the annihilating ideal
of a highest weight representation of $\fing$ \cite{Duf77},
it suffices to classify simple
highest weight representations of $\Zhu(\Vs{k})$,
or equivalently,
to classify
 simple $\Vs{k}$-modules in the category $\BGG$
of $\affg$.

Let $L(\lam)$ be the irreducible highest weight representation
of $\affg$ with highest weight $\lam$.
\begin{MainTh}
Let $k$ be an admissible number\footnote{A complex number $k$ is called
admissible if $\Irr{k\Lam_0}$ is an admissible representation.
},
$\lam$ a  weight of $\affg$ of level $k$.
Then $L(\lam)$ is a module over $\Vs{k}$ if and only if
it is an admissible representation   whose integral root system is isomorphic to
that of  $\Vs{k}$.
In particular
any $\affg$-module from 
the category $\BGG$
is an $\Irr{k\Lam_0}$-module if and only if 
it is a direct sum of 
such admissible representations of $\affg$ of level $k$.
 
\end{MainTh}

The second statement of
Main Theorem
was conjectured by
Adamovi\'c and Milas \cite[Conjecture 3.5.7]{AdaMil95}.
It
  has been proved 
in some special cases:
for type $C_{\ell}^{(1)}$ admissible half integer levels
by Adamovi\'c \cite{Ada94};
for $\wh{\mf{sl}}_2$
by Adamovi\'c and Milas \cite{AdaMil95}, Dong, Li and Mason
\cite{DonLiMas97}
and Feigin and Malikov \cite{FeiMal97};
for some cases in type $A_{\ell}^{(1)}$, $B_{\ell}^{(1)}$
admissible half integer levels by Per{\v{s}}e
\cite{Per07, Per08};
for a few cases in
type $G_2^{(1)}$
admissible one-third integer levels
by
Axtell and Lee
\cite{AxtLee11}.
These works are based on the explicit
computation of the singular vector of $\Vg{k}$
which generates the maximal ideal.
Our method in this article is 
completely  different.

Let us  explain  the  outline of the proof of Main Theorem
briefly.
By  (\ref{eq:zhu-of-simple})
simple highest weight $\Vs{k}$-modules
are classified by 
Joseph's {\em characteristic variety}
\cite{Jos77}
$\zeroset(I_k)$ of $I_k$,
which is a Zariski closed subset of
the dual
$\finh^*$ of the Cartan subalgebra $\finh$
of $\fing$ (Proposition \ref{Pro:zhu-vs-joseph}).
We deduce the ``if'' part  of Main Theorem
 from
a result of
Frenkel and Malikov \cite{FreMal97,MalFre99}
which states
 that every $G$-integrable admissible representation
at level $k$ is a 
module over $\Vs{k}$,
together with
an affine analogue  of the Duflo-Joseph Lemma
(Lemma \ref{Lem:affin-Duflo-Joseph}) for  
characteristic varieties.
We prove the ``only if'' part of Main Theorem 
by 
reducing to the
$\wh{\mf{sl}}_2$-cases \cite{AdaMil95}
using the {\em semi-infinite restriction functor}
$H^{\semiinf+i}(\finm^{(i)}[t,t\inv],?)$
studied in \cite{A-BGG},
where $\finm^{(i)}$ is the nilradical of a minimal parabolic subalgebra
of $\fing$,
see \S \ref{Lem:top-part-of-cohomology}
for the details.

It should be mentioned that there is another variety naturally associated with
$\Zhu(\Vs{k})$,
that is,
 the zero set
$V(\gr I_k)$ of the associated graded ideal $\gr I_k$ of
$\C[\fing^*]$.
We have a  surjection
\begin{align}\label{eq:surjection}
R_{\Vs{k}}
\twoheadrightarrow \gr \Zhu(\Vs{k})=\C[\fing^*]/\gr I_k
\end{align}of Poisson algebras,
where
$R_{\Vs{k}}$ is Zhu's $C_2$-algebra of 
$\Vs{k}$,
see \cite[Proposition 3.3]{ALY}.
Hence
$V(\gr I_k)$ is contained in 
the {\em associated variety} \cite{Ara12}
$X_{\Vs{k}}=\on{Specm} R_{\Vs{k}}$.
Although (\ref{eq:surjection})
is not an isomorphism in general,
in a subsequent paper \cite{A2012Dec} we prove the following:
\begin{Th}\label{Th;subsquence}
Let $k$ be an admissible number,
with $k+h\che=p/q$,
$p,q\in \N$, $(p,q)=1$.
Then (\ref{eq:surjection})
induces the  isomorphism
of varieties
\begin{align*}
V(\gr I_k)\cong X_{\Vs{k}}.
\end{align*}
Namely \cite{Ara09b},
we have 
\begin{align*}
V(\gr I_k)
\cong \begin{cases}
   \{x\in \fing; (\ad x)^{2q}=0\}&\text{if }(r\che, q)=1,\\
\{x\in \fing; \pi_{\theta_s}(x)^{2q/r\che}=0\} &\text{if }(r\che,q)=r\che,
\end{cases}\end{align*}
which is an
irreducible $G$-invariant
subvariety  
 of $\fing^*\cong \fing$.
Here $r\che$ is the lacing number of $\fing$,
$\theta_s$ is the highest short root of $\fing$
and $\pi_{\theta_s}$ is the irreducible finite-dimensional
representation of $\fing$
with highest weight $\theta_s$.
\end{Th}

Although 
they themselves are 
not  rational 
in the usual 
sense\footnote{However Main Theorem implies that
admissible affine  vertex algebras are rational 
in the sense of \cite{DonLiMas97}.
Also, the fact \cite{Ara09b}
that $X_{L(k\Lam_0)}$ is contained in the nilpotent cone
of
$\fing$
implies that they are
 $C_2$-cofinite in the sense of \cite{DonLiMas97}.}
it has been  conjectured 
\cite{FKW92,KacWak08}
that
admissible affine
vertex algebras 
produce 
rational 
$W$-algebras 
in many cases
by the method of the 
 (generalized)  quantum Drinfeld-Sokolov reduction.
By applying  the result in this article
in a subsequent paper \cite{A2012Dec}
we prove the rationality 
of 
all the minimal series principal $W$-algebras \cite{FKW92}.

\smallskip

This paper is organized as follows.
In \S \ref{section:Affine vertex algebras and Joseph's characteristic
varieties}
we establish the relationship between Joseph's characteristic
varieties and the representation theory of affine vertex algebras
and prove an affine analogue of
the Duflo-Joseph Lemma (Lemma \ref{Lem:affin-Duflo-Joseph}).
In \S \ref{section:Kac-Wakimoto Admissible representations}
we prove  the ``if part'' of Main Theorem.
In \S \ref{section:Semi-infinite restriction functors}
we prove the ``only if part'' of Main Theorem.
In appendix 
we apply Main Theorem to prove a semi-infinite analogue of 
Kostant's generalized  Borel-Weil-Bott Theorem \cite{Kos61}
for admissible representations announced in \cite{A-BGG}.

\subsection*{Acknowledgments}
Some part of this work was done
while the author was visiting
 Weizmann Institute, Israel, in May 2011,
Emmy Noether Center in Erlangen,
Germany
 in June 2011, 
Isaac Newton Institute for Mathematical Sciences,
UK, in 2011,
The University of Manchester,
University of Birmingham,
The University of Edinburgh,
 Lancaster University,
York University, UK,  in November 2011,
Academia Sinica, Taiwan, in December 2011.
He is grateful to those institutes
for their hospitality.

\section{Affine vertex algebras and Joseph's characteristic varieties}
\label{section:Affine vertex algebras and Joseph's characteristic varieties}

Let $\fing$ be  a complex simple Lie algebra
of rank $l$.
Fix
a triangular decomposition
\begin{align*}
\fing=\finn_-\+ \finh\+ \finn_+,
\end{align*}
with  
a Cartan subalgebra $\finh$ of $\fing$.
We will often identify $\finh$     with $\finh^*$
using
  the normalized invariant bilinear form  $(~|~)$  of $\fing$.
Let $\Delta$, $\Delta_+$,
$\Delta_-$,
$\Pi=\{\alpha_1,\dots,\alpha_l\}$
be the 
sets of the roots, positive roots,
 negative roots,
and simple roots  of $\fing$,
respectively.
Also, let
$\theta$ be the highest root of $\fing$,
$\theta_s$ the highest short root,
$\rho$  the half sum of the positive roots,
 $\sW=\bra s_{\alpha};
\alpha\in \Delta \ket\subset \Aut \finh^*$ the Weyl group of $\fing$.
Here $s_{\alpha}(\lam)=\lam-\bra \lam,\alpha\che\ket \alpha$,
$\alpha\che=2\alpha/(\alpha|\alpha)$.
We set $s_i=s_{\alpha_i}$ for $i=1,\dots,l$.
Let $w\circ \lam=w(\lam+\rho)-\rho$
for $w\in W$, $\lam\in \dual{\finh}$.

Denote by $\sBGG$  the Bernstein-Gelfand-Gelfand
category $\BGG$ of $\fing$,
 by $\sIrr{\lam}$ 
 the irreducible highest weight representation of 
$\fing$
with highest weight $\lam\in \dual{\finh}$.

Set
$U(\fing)^{\finh}:=\{u\in U(\fing);
[h,u]=0\ \text{for all }h\in \finh\}$,
and 
let
\begin{align*}
 \Upsilon:U(\fing)^{\finh}\ra U(\finh)
\end{align*}
be the restriction of 
the projection
$
U(\fing)=U(\finh)\+  (\finn_- U(\fing)+U(\fing)\finn_+)
\ra U(\finh)
$ to $U(\fing)^{\finh}$.
One knows that
$\Upsilon$ is an algebra homomorphism.

For   a two-sided ideal $I$ of 
$U(\fing)$,
the  {\em characteristic variety} \cite{Jos77}  (without the $\rho$-shift)
of $I$
is defined as
\begin{align*}
 \zeroset(I)=\{\lam\in \dual{\finh};
p(\lam)=0\text{ for all }p\in \HC(I^{\finh})\},
\end{align*}
where $I^\finh=I\cap U(\fing)^{\finh}$.
\begin{Lem}\label{lemma:property-of-characteristic-vareity}
For  $\lam\in \dual{\finh}$,
$ \lam \in \zeroset(I)$ if and only if 
$ I \sIrr{\lam}=0$.
\end{Lem}
\begin{proof}
 The ``if part'' is obvious.
To see the converse,
suppose that
 $I\sIrr{\lam}\ne 0$.
Then  $I\sIrr{\lam}=\sIrr{\lam}$ since $\sIrr{\lam}$ is simple.
It follows that there exists $a\in I^{\finh}$
such that $v_{\lam}=a v_{\lam}=\HC(a)(\lam)v_{\lam}$,
where $v_{\lam}$ is the highest weight vector of $\sIrr{\lam}$.
Hence $\lam\not \in \mc{V}(I)$.
\end{proof}
By  Lemma \ref{lemma:property-of-characteristic-vareity},
the characteristic variety
 $\zeroset(I)$  of $I$
classifies simple $U(\fing)/I$-module in the 
category $\sBGG$.

The following fact is useful for us. 
\begin{Lem}[{\cite{Duf77}, \cite[Lemma 2]{Jos77}}]
\label{Lemma:Duflo}
Let $I$ be a two-sided ideal of $U(\fing)$,
$\lam\in \zeroset(I)$.
Suppose that
$\bra \lam+\rho,\alpha_i\che\ket \not \in \N$ for 
$\alpha_i\in \Pi$.
Then
$s_i\circ \lam\in \zeroset(I)$.
\end{Lem}

Let \begin{align*}
\affg=\fing[t,t\inv]\+ \C K
\end{align*}
be the affine Kac-Moody algebra
associated with $\fing$  as in Introduction.
The commutation relations of $\affg$
are given by
\begin{align*}
& [xt^m,yt^n]=[x,y]t^{m+n}+ m(x|y)\delta_{m+n,0}K,
\quad [K,\affg]=0
\end{align*}
for $x,y\in \fing$,
$m,n\in \Z$.
Let
\begin{align*}
 \affg=\affn_-\+ \affh\+ \affn_+
\end{align*}
be the triangular decomposition of
$\affg$,
where
$\affn_-=\finn_-\+ \fing[t\inv]t\inv$,
$\affh=\finh\+\C K$,
$\affn_+=\finn_+\+ \fing[t]t$.
Let $\dual{\affh}=\finh^*\+ \C \Lam_0$ be the dual
of 
$\affh$,
where $\Lam_0(K)=1$, $\Lam_0(\finh)=0$.
For $\lam\in \dual{\affh}$,
denote by  $\bar \lam\in \dual{\finh}$ the restriction of $\lam$ to $\finh$.

Let $\wh{\Delta}^{re}$
be the
set of
real roots of
$\affg$
in the dual $\tilde{\finh}^*
=\dual{\affh}\+ \C \delta$ of the extended Cartan subalgebra
$\tilde{\finh}=\affh\+ \C D$,
$\wh{\Delta}^{re}_+$
the set of positive real roots,
$\wh{\Pi}=\{\alpha_0,\alpha_1,\dots,\alpha_{l}\}$
the
 the set of simple roots of $\affg$,
where
$\alpha_0=-\theta+\delta$.

Let $\affW=\sW\ltimes Q\che$,
the Weyl group of $\affg$,
which is generated by $s_{0},s_1,\dots,s_{l}$.
Here 
 $Q\che=\sum_{\alpha\in \Delta}\Z\alpha\che$,  the coroot lattice of
 $\fing$,
and  $s_0$ is the reflection corresponding to $\alpha_0$.
Let $\eW=\sW\ltimes P\che$,
the extended Weyl group of $\affg$,
where $P\che$  is
 the coweight lattice
of $\fing$.
We write
$t_{\lam}$ for 
the element of $\affW$
corresponding to $\lam \in P\che$.
It holds that
\begin{align*}
\eW=\eW_+\ltimes \affW,
\end{align*}
where
$\eW_+$ is the subgroup
of
$\eW$ consisting of elements which fix the set $\wh{\Pi}$.
Write
$\theta=\sum_{i=1}^la_i \alpha_i$
and set
$J=\{i\in \{0,1,\dots, l\};
a_i=1\}$.
The group $\eW_+$ is described as 
\begin{align}
 \eW_+=\{t_{\bar \Lam_j}w_j
; j\in J\},
\label{eq:elements-of-eW-of-Dynkin-auto}
\end{align}
where 
$\bar \Lam_j$ is the $j$-th fundamental weight of $\fing$
and $w_j$ is the unique element of $\sW$
which fixes the set $\{\alpha_1,\dots,\alpha_{l},-\theta\}$
and $w_j(-\theta)=\alpha_j$.
We have the isomorphism
$\eW_+\cong P\che/Q\che$,
$t_{\bar \Lam_j}w_j\mapsto \bar \Lam_j+Q\che$.

Let $k\in \C$.
Denote by $U_k(\affg)$  the quotient of the universal enveloping algebra
$U(\affg)$ 
by the ideal generated by $K-k\id$,
and let
\begin{align*}
 \widetilde{U}_k(\affg):=\lim_{\longleftarrow\atop N}U_k(\fing)/
(U_k(\fing)\fing[t]t^N),
\end{align*}
the {\em  completed universal enveloping  algebra} of $\affg$
(\cite[4.2.1]{FreBen04})
at level $k$.

A $\affg$-module $M$ is said to be {\em of level} $k$
if $K$ acts as the  scalar $k$.
It is called {\em smooth}
if
 $(xt^n)m=0$ for 
any $x\in\affg$, $m\in M$
and a sufficiently large $n$.
A smooth $\affg$-module $M$ 
of level $k$
is naturally considered as a  continuous $\widetilde{U}_k(\affg)$-module.

Let 
$\Vg{k}$
be the universal affine vertex algebra associated with $\fing$
at level $k$
as in Introduction.
By definition, 
\begin{align*}
 \Vg{k}=U(\affg)\otimes_{U(\fing[t]\+\C K)}\C_k
\end{align*}
as a $\affg$-module,
where
 $\C_k$ is the one-dimensional
representation of 
$\fing[t]\+ \C K$ on which 
$\fing[t]$ acts trivially and $K$ acts as a  multiplication by $k$.
Note that by the PBW theorem
we have
$\Vg{k}\cong U(\fing[t\inv]t\inv)$ 
as  vectors spaces.
Define 
the linear map
\begin{align*}
\tilde Y: \Vg{k}\ra \widetilde{U}_k(\affg)[[z,z\inv],\quad
a\mapsto  \tilde Y(a,z)=\sum_{n\in\Z}a_{(n)}z^{-n-1} 
\end{align*}by
\begin{align*}
 \tilde Y((x_1t^{-n_1-1})\dots (x_r t^{-n_r-1})\1,z)
=\frac{1}{n_1!\dots n_r!}
:\partial_z^{n_1}x_1(z)\dots \partial_z^{n_r}x_r(z):,
\end{align*}
where $\1=1\otimes 1$,
$x(z)=\sum_{n\in \Z}(xt^n)z^{-n-1}$ for $x\in \fing$
and $:~:$ denotes 
the normally ordered product
(see \cite[3.1]{Kac98}).
The space $\Vg{k}$ is 
equipped with the  vertex algebra structure
whose
 state-field correspondence
$Y:\Vg{k}\ra  (\End \Vg{k})[[z,z\inv]]$
is given by
the composition of $\tilde{Y}$ with 
the action map $\UU_k(\affg)\ra \End \Vg{k}$,
see \cite{Kac98,FreBen04} for the details.
We have
\begin{align}
& [a_{(m)}, b_{(n)}]=\sum_{i\geq 0}\begin{pmatrix}
				   m\\ i
				  \end{pmatrix}(a_{(i)}b)_{(m+n-i)},
\label{eq:commu-relation}
\\
& (a_{(m)}b)_{(n)}=\sum_{i\geq 0}(-1)^i\begin{pmatrix}
					m\\i
				       \end{pmatrix}
(a_{(m-i)}b_{(n+i)}-(-1)^r b_{(m+n-i)}a_{(i)})
\end{align}
for $a,b\in \Vg{k}$,
$m,n\in \Z$,
in $\widetilde{U}_k(\affg)$
by \cite[Lemma 4.3.2]{FreBen04}.

For a 
 smooth $\affg$-module of level $k$,
let $Y^M:\Vg{k}\ra (\End M)[[z,z\inv]]$
be
the composition of $\tilde{Y}$ with 
the action map $\widetilde{U}_k(\affg)\ra \End M$.
This gives $M$ a 
$\Vg{k}$-module structure.
Conversely any $\Vg{k}$-module can be considered as a smooth
$\affg$-module of level $k$.
Hence a $\Vg{k}$-module is the same as a smooth $\affg$-module of level $k$.

We assume that the 
level $k$ is non-critical,
that is,
$k\ne -h\che$,
where $h\che$ is the dual Coxeter number of $\fing$.
The vertex algebra $\Vg{k}$ is conformal by the Sugawara construction;
let 
\begin{align*}
L(z)=\sum_{n\in \Z}L_nz^{-n-2}=\tilde Y(\omega,z),
\end{align*}
where $\omega=\frac{1}{2(k+h\che)}\sum_{i=1}^{\dim \fing}x_i t\inv x^i
t\inv \1$,
$\{x_i\}$ is a basis of $\fing$,
$\{x^i\}$ is the dual basis,
and
$\1=1\otimes 1\in \Vg{k}$.
We have
\begin{align*}
 &[L_m,L_n]=(m-n)L_{m+n}+\frac{m^3-m}{12}\delta_{n+m,0}\frac{k \dim
 \fing}{k+h\che},\\
&[L_m, x t^n]=- nxt^{n-m}\quad\text{for }x\in\fing,\ n\in \Z.
\end{align*}

For a smooth
$\affg$-module $M$ of level $k$
and $d\in \C$,
set
\begin{align*}
M_{d}=\{m\in M;(L_0-d)^r  m=0\ \text{ for }r\gg 0\}.
\end{align*}
Note that each
$M_{d}$ is a $\fing$-submodule of $M$.
We say
$M$ is graded if $M=\bigoplus_{d\in \C}M_d$;
positively graded if there exists
$d_{top}\in \C$ such that
$M_{d_{top}}\ne 0$
and 
$M=\bigoplus_{n\in \Z_{\geq 0}}M_{d_{top} +n}$.
If this is the case we often write
$M_{top}$ for $M_{d_{top}}$.

The
module
$\Vg{k}$
 is positively graded:
\begin{align*}
 \Vg{k}=\bigoplus_{\Delta=0}^{\infty}\Vg{k}_{\Delta}.
\end{align*}
For a homogeneous element $a\in \Vg{k}$,
the $L_0$-eigenvalue of $a$ is called the {\em conformal weight}
of $a$ and denoted by $\Delta_a$.

We denote by  $\mathscr{L}(\Vg{k})$  Borcherds' Lie algebra associated with $\Vg{k}$:
\begin{align*}
 \mathscr{L}(\Vg{k})=\Vg{k}\otimes \C[t,t^{-1}]/(L_{-1}\otimes1 +1\+ \frac{d}{dt})(\Vg{k}\otimes \C[t,t\inv]).
\end{align*}
The Lie bracket of $\mathscr{L}(\Vg{k})$ is given by
\begin{align*}
 [a_{\{m\}}, b_{\{n\}}]=\sum_{i\geq 0}\begin{pmatrix}
				   m\\ i
				  \end{pmatrix}(a_{(i)}b)_{\{m+n-i\}}
\end{align*}
for $a,b\in V$, $m,n\in \Z$,
where 
$a_{\{n\}}$ denotes the image of $a\otimes t^n$ in $\mathscr{L}(\Vg{k})$.
The Lie algebra
 $\mathscr{L}(\Vg{k})$ is graded:
\begin{align*}
 \mathscr{L}(\Vg{k})=\bigoplus_{\Delta}\mathscr{L}(\Vg{k})_{\Delta},
\end{align*}
where 
 $\deg a_{\{m\}}=m-\Delta_a+1$ 
for a homogeneous vector $a\in \Vg{k}$.

Note that
$(xt^{-1}\1)_{\{m\}}$, $x\in \fing$, $m\in \Z$,
generates the Kac-Moody algebra  $\affg$ at level $k$,
and hence, there is an adjoint action of $\affg$
on $\mathscr{L}(\Vg{k})$.

By \eqref{eq:commu-relation},
there is a natural Lie algebra homomorphism
\begin{align*}
 \mathscr{L}(\Vg{k})\ra  \UU_k(\affg),
\quad a_{\{m\}}\mapsto a_{(m)},
\end{align*}
which is known to be injective (\cite[4.2.7]{FreBen04}). 
Below we consider
$\mathscr{L}(\Vg{k})$ as a graded Lie subalgebra
of $\UU_k(\affg)$
and write $a_{(m)}$ for $a_{\{m\}}$.
Clearly,
$\mathscr{L}(\Vg{k})$
is a $\ad \affg$-submodule of 
$\UU_k(\affg)$.

For a $\Z$-graded vertex algebra $V$,
let $A(V)$ be Zhu's algebra of $V$.
By \cite[Theorem 3.1.1]{Zhu96},
one knows that  $A(\Vg{k})\isomap U(\fing)$
as algebras.
This isomorphism 
 may be constructed as follows:
Consider the decomposition
$U_k(\affg)=U(\fing)\+ (\fing[t\inv]t\inv
U_k(\affg)+U_k(\affg)\fing[t]t)$.
This induces the decomposition
\begin{align}
 \UU_k(\affg)=U(\fing)\+\overline{(\fing[t\inv]t\inv
U_k(\affg)+U_k(\affg)\fing[t]t)},
\end{align}
where,
for a subspace $M$
of $\UU_k(\affg)$,
 $\overline{M}$ denotes the 
  closure of 
$M$.
Let $\Phi:\UU_k(\affg)\ra U(\fing)$
be the projection with respect to the above decomposition.
\begin{Lem}\label{Lem:Zhu-iso-FZ}
 The correspondence $[a]\mapsto \Phi(o(a))$
gives the algebra isomorphism
$A(\Vg{k})\isomap U(\fing)$,
where
$o: \Vg{k}\ra \UU_k(\affg)$ is the linear map
defined by  $o(a)=a_{(\Delta_a-1)}$
for a homogeneous element $a$ of $V$.
\end{Lem}
\begin{proof}
 The fact that
the above map
 is an algebra homomorphism
follows from the formula
\begin{align}
 o(a)o(b)\equiv o(a * b) \pmod{\overline{(\fing[t\inv]t\inv
U_k(\affg)+U_k(\affg)\fing[t]t)}},
\label{eq:Zhu-alg-hom}
\end{align}
where $a * b=\sum_{i\geq 0}\begin{pmatrix}
			    \Delta_a\\i
			   \end{pmatrix}a_{(i-1)}b$
for a homogeneous element $a\in \Vg{k}$
(see \cite[Theorem 2.1.2]{Zhu96}).
The assertion follows by
recalling the proof of 
\cite[Theorem 3.1.1]{FreZhu92}.
\end{proof}

Let $\BGG_k$ be the full subcategory of
 $\affg$-modules of level $k$  consisting of objects $M$
 on which
(1) $L_0$ acts locally finitely,
(2) $\affn_+$ acts locally nilpotently,
(3) $\affh$ acts semisimply.
An object of $\BGG_k$ is obviously smooth.
Note that
$\Vg{k}$ is an object of $\BGG_k$.
Let
$L(\lam)$  be  the irreducible highest weight representation
of $\affg$ with highest weight $\lam\in \dual{\affh}$.
Clearly,
$L(\lam)$ is positively graded and 
\begin{align*}
L(\lam)_{top}\cong
\sL_{\slam}^{\fing}
\end{align*}as $\fing$-modules.
The category
 $\BGG_k$  is naturally regarded as a full subcategory of the
category of $\Vg{k}$-modules.

\smallskip

Let $N_k$ be the 
unique maximal  ideal of
 $\Vg{k}$.
Then
\begin{align*}
\Vg{k}/N_k\cong L(k\Lam_0)
\end{align*}
as $\affg$-modules,
where
 $\Lam_0\in \dual{\affh}$
is the $0$-th fundamental weight of $\affg$:
$\Lam_0(K)=1$, $\Lam_0(\finh)=0$.
The vertex algebra $L(k\Lam_0)$
is called the {\em (simple) affine vertex
algebra}
associated with $\fing$
at level $k$.
By definition, a
$\Vg{k}$-module $M$ is
a $\Vs{k}$-module if and only if 
$a_{(n)}$ 
annihilates 
$M$ for all  $a\in N_k$, $n\in \Z$.

Let $\mathscr{L}(N_k)$
 be the image of $N_k$ in $\mathscr{L}(\Vg{k})$.
Then $\mathscr{L}(N_k)$ is a graded ideal of 
$\mathscr{L}(\Vg{k})$:
$\mathscr{L}(N_k)=\bigoplus_{\Delta\in \Z}\mathscr{L}(N_k)_{\Delta}$,
where $\mathscr{L}(N_k)_{\Delta}=\mathscr{L}(N_k)\cap \mathscr{L}(\Vg{k})_{\Delta}$.
We have
$\mathscr{L}(L(k\Lam_0))\cong \mathscr{L}(\Vg{k})/\mathscr{L}(N_k)$.
Set
\begin{align*}
 I_k:=\Phi(\mathscr{L}(N_k)_0)\subset U(\fing).
\end{align*}
By \eqref{eq:Zhu-alg-hom},
it follows that $I_k$ is an two-sided ideal of $U(\fing)$.

The following follows immediately from the definition of
Zhu's algebra \cite{FreZhu92}.
\begin{Lem}\label{Lem:Zhu-iso-L}
 The isomorphism $A(\Vg{k})\isomap U(\fing)$ in Lemma
 \ref{Lem:Zhu-iso-FZ}
induces the isomorphism
$A(L(k\Lam_0))\isomap U(\fing)/I_k$.
\end{Lem}
By  Zhu's theorem \cite{Zhu96},
the correspondence 
\begin{align*}
M\mapsto M_{top}
\end{align*}
gives a bijection between the set
of isomorphism classes of simple positively graded
$L(k\Lam_0)$-modules and that of simple $A(L(k\Lam_0))$-modules.
Hence 
Lemma  \ref{lemma:property-of-characteristic-vareity}
and 
Lemma \ref{Lem:Zhu-iso-L}
give the following assertion.
\begin{Pro}\label{Pro:zhu-vs-joseph}
For a weight  $\lam\in \dual{\affh}$ of level $k$,
$ \Irr{\lam}$ is a module over $\Vs{k}$ if and only if
$ \bar \lam\in \zeroset(I_k)$.
(Recall that $\bar \lam$ denotes the restriction of $\lam$ to $\finh$.) 
\end{Pro}

\begin{Lem}\label{Lem:Joseph}
Let 
$\lam\in \dual{\affh}$
be a weight of level $k$
such that
$\bar \lam\in \zeroset(I_k)$.
Suppose that
$\bra \lam+\affrho, \alpha_i\che\ket \not \in \N$
for $\alpha_i\in \wh{\Pi}$.
Then $\overline{s_i\circ \lam}\in \zeroset(I_k)$.
\end{Lem}
 \begin{proof}
The case
$i=1,\dots, l$ is the statement of Lemma \ref{Lemma:Duflo}.
We shall prove the assertion for $i=0$
based on the argument in the proof of 
\cite[Lemma 2]{Jos77}.
Let 
 $\mf{sl}_2^{(0)}=\haru_{\C}\{e_0,f_0,\alpha_0\che\}\subset \affg$ be the copy of
  $\mf{sl}_2(\C)$,
where $e_0$ and $f_0$ are root vectors of root $\alpha_0$ and
  $-\alpha_0$,
respectively.
Set 
\begin{align*}
\affb^{(0)}=\mf{sl}_2^{(0)}+(\affh\+\affn_+)
=\affl\+ \affm,
\end{align*}
the minimal parabolic subalgebra of $\affg$,
where
$\affl$  is its Levi subalgebra
and  $\affm$
is  its nilradical.
We have
$\affl=\mf{sl}_2^{(0)}+\affh_0^{\bot}$
and
$\affm=\bigoplus\limits_{\alpha\in\widehat{\Delta}^{re}_+\atop \alpha\ne \alpha_0}\affg_{\alpha
}$,
where $\affh^{\bot}_0$
is
the orthogonal complement  
of $\C \alpha_0\che$ in $\affh$
and
$\affg_{\alpha}$ denotes the root space of $\affg$ of root $\alpha$.
Denote by 
$\affm_-$ the opposite subalgebra 
of $\affm$,
so  that
 $\affg=\affm_-\+\affl\+\affm_+$.
Then we have the decomposition
$U_k(\affg)=U_k(\affl)\+ (\affm_- U_k(\affg)+U_k(\affg)\affm)$,
where
$U_k(\affl)=U(\affl)/(K-k)U(\affl)$.
This induces the decomposition
\begin{align*}
 \UU_k(\affg)=U_k(\affl)\+
\overline{(\affm_-\UU_k(\affg)\+
\UU_k(\affg)\affm)},
\end{align*}
Let $\wh\Upsilon^{(0)}:
\UU_k(\affg)\ra U_k(\affl)
$ be the projection with respect to this
  decomposition.
Set
\begin{align*}
  I:=\wh\Upsilon^{(0)}(\mathscr{L}(N_k)_0),
\end{align*}
which is an $\ad \affl$-invariant subspace of $U_k(\affl)$.
Note that
we have 
\begin{align}
 \Upsilon(I_k)=\gamma^{(0)}(I),
\label{eq:connection-to-Joseph}
\end{align}
where
$\gamma^{(0)}:U(\affl)\ra U(\finh)$ is the projection
defined by the decomposition
$U_k(\affl)=U(\finh)\+ (f_0U_k(\affl)
+ U_k(\affl)e_0)$.

Since it is a direct sum of semisimple finite-dimensional representations
 $I$ is generated by 
$I^{e_{0}}:=\{u\in I; [e_0,u]=0\}$ as an $\ad \affl$-module,
and by Kostant's separation theorem $I^{e_{0}}$ is spanned by
the vectors of the form
\begin{align*}
 e_0^n b_n \quad\text{with }n\in \Z_{\geq 0},  \ b_n\in  Z_0,
\end{align*}
where 
$Z_0$ is the subalgebra of $U_k(\affl)$ generated by 
the quadratic Casimir element  $\Omega_0\in U(\mf{sl}_2^{(0)})$ 
and $\affh^{\bot}_0$.

Now suppose that
$e_0^n b_n\in I^e$,
$n\in \Z_{\geq 0}$,
$b_k\in Z_0$, annihilates $L(\lam)$.
Then $e_0^n f_0^n b_n$ also annihilates $L(\lam)$.
Since
\begin{align*}
 e_0^nf_0^n b_nv_{\lam}=
\lam(\alpha_0\che)(\lam(\alpha_0\che)-1)\dots (\lam(\alpha_0\che)-n+1)b_nv_{\lam},
\end{align*}
where  $v_{\lam}$ is the highest weight vector of $L(\lam)$,
the assumption implies that
$\lam$ is a zero of 
$\gamma^{(0)}(b_n)$,
and so is 
$s_0\circ \lam $.
It follows that
$I^{e_0}$ is spanned by the vectors
$e_0^n b_n$, $n\in \Z_{\geq 0}$, $b_n\in Z_0$,
such that
$\gamma^{(0)}(b_n)(s_0\circ \lam)=0$.
By
 \eqref{eq:connection-to-Joseph},
we conclude 
 that
$s_0\circ \lam$ is a zero of $\HC(I_k)$.
 \end{proof}
\begin{Lem}\label{Lem:dynkin-auto}
Let 
$\lam\in \dual{\affh}$
be a weight of level $k$ such that
$\bar \lam\in \zeroset(I_k)$.
Then
 $\overline{w\circ \lam}\in \zeroset(I_k)$
for any $w\in \widetilde{W}_+$.
\end{Lem}
  \begin{proof}
Let $w\in \widetilde{W}_+$.
Then $w\inv=t_{\bar \Lam_j}w_j$ for some $j\in J$,
see \eqref{eq:elements-of-eW-of-Dynkin-auto}.
Let $\tilde{w}_j$ be a Tits lifting of $w_j$ to the adjoint group
of $\fing$.
Also,
let
$\Delta(\bar \Lam_j,z)$
be 
 Li's delta operator
\cite{Li97}
corresponding to $\bar\Lam_j$.
One can 
obtain
a new simple
$L(k\Lam_0)$-module
from $L(\lam)$
by
twisting
the action 
$a\mapsto Y^{L(k\Lam_0)}(a,z)$
on $L(\lam)$
by $a\mapsto Y^{L(k\Lam_0)}(\Delta(\bar \Lam_j,z)
\tilde{w}_ja,z)$.
Since
the twisting
the action
by 
$\Delta(\bar \Lam_j,z)$ induces the 
automorphism of $\affg$
corresponding
to $t_{\bar \Lam_j}$
(see \cite[(3.15)--(3.17)]{Li97}),
one finds that the
resulting $L(k\Lam_0)$-module is 
in the category $\BGG_k$
and
isomorphic to $L(w\circ \lam)$ 
as $\affg$-modules.
This completes the proof.
  \end{proof}
\begin{Lem}\label{Lem:affin-Duflo-Joseph}
Let 
$\lam\in \dual{\affh}$
be a weight of level $k$ such that
$\bar \lam\in \zeroset(I_k)$,
and let $w\in \eW$.
Suppose that
$\bra \lam+\affrho, \alpha\che\ket \not \in \N$
for all $\alpha\in \wh{\Delta}^{re}_+
\cap w\inv (\wh{\Delta}^{re}_-)$.
Then $\overline{w\circ \lam}\in \zeroset(I_k)$.
\end{Lem}
 \begin{proof}
By Lemma \ref{Lem:dynkin-auto}
we may assume that
$w\in \affW$.
We proceed
by induction on the length
$\ell(w)$ of $w$.
The case that  $\ell(w)=1$
has been proved
in Lemma \ref{Lem:Joseph}.
So let $\ell(w)>1$.
Write $w=s_{i}y$
with 
$i\in \{0,1,\dots,l\}$,
$y\in \affW$,
$\ell(w)=\ell(y)+1$.
Then
\begin{align*}
\wh{\Delta}^{re}_+
\cap w\inv (\wh{\Delta}^{re}_-)=
\{y\inv(\alpha_i)\}\sqcup 
(\wh{\Delta}^{re}_+
\cap y\inv (\wh{\Delta}^{re}_-)).
\end{align*}
Hence by the induction hypothesis
we have $\overline{y\circ \lam}\in \zeroset(I_k)$.
Since
$\bra y\circ \lam+\wh{\rho},\alpha_i\che\ket=\bra
  y(\lam+\widehat{p}),\alpha_i\che)
=\bra \lam+\wh{\rho},y\inv(\alpha_i\che)\ket\not\in \N$,
 Lemma \ref{Lem:Joseph}
gives that
$\overline{w\circ \lam}\in \zeroset(I_k)$
as 
required.
 \end{proof}
\section{Kac-Wakimoto Admissible representations}
\label{section:Kac-Wakimoto Admissible representations}
For $\lam\in \dual{\affh}$,
let $\wh\Delta(\lam)$
and $\affW(\lam)$ be its integral root system
and its integral Weyl group,
respectively:
 \begin{align*}
& \wh\Delta(\lam)
=\{\alpha\in \wh{\Delta}^{re}; \bra
 \lam+\affrho,\alpha\che\ket
\in \Z\},
&\affW(\lam)=\bra s_{\alpha}; \alpha\in\wh \Delta(\lam) \ket.
\end{align*}
Let
$\wh\Delta(\lam)_+=\wh{\Delta}(\lam)\cap \wh{\Delta}^{re}_+$,
the set of positive roots of 
$\wh{\Delta}(\lam)$
and 
$\Pi(\lam)\subset \wh{\Delta}(\lam)_+$, the set of simple roots.

A weight $\lam\in \dual{\affh}$ is called {\em admissible} if
\begin{enumerate}
 \item $\lam$ is regular dominant, that is,
$\bra \lam+\affrho,\alpha\che\ket \not \in \{0,-1,-2,\dots\}$,
 \item $\Q\wh\Delta(\lam)=\Q \wh{\Delta}^{re}$.
\end{enumerate}
An  {\em admissible number} (for $\affg$) is 
a complex number $k$ such that
$k\Lam_0$   is admissible.
\begin{Pro}[\cite{KacWak89,KacWak08}]
\label{Pro:admissible number}
A complex  number  $k$ is admissible if and only if 
 \begin{align*}
  k+h\che=\frac{p}{q}
 \quad \text{with }p,q\in \N,\ (p,q)=1,\ p\geq 
\begin{cases}
h\che&\text{if }(r\che, q)=1\\
h&\text{if }(r\che,q)=r\che,
\end{cases}
 \end{align*}
where 
$h$  is
 the the Coxeter number of $\fing$
and $r\che$ is the lacing number of $\fing$,
that is, the maximal number of the edges in the Dynkin diagram of $\fing$.
If this is the case
we have
 $\wh{\Pi}(k\Lam_0)=\{\dot{\alpha_0},\alpha_1,\alpha_2,\dots,\alpha_l\}$,
where 
\begin{align*}
  \dot{\alpha_0}= 
\begin{cases}
-\theta+q\delta&\text{if }(r\che,q)=1
\\
 -\theta_s+\frac{q}{r\che}\delta&\text{if }(r\che ,q)=r\che.
\end{cases}
\end{align*}
\end{Pro}

For an admissible number $k$
let $Pr^k$ be the set of admissible weights
$\lam$ such that $\wh{\Delta}(\lam)\cong \wh{\Delta}(k\Lam_0)$
as root systems.
Set
\begin{align*}
 Pr_{\Z}^k=Pr^k\cap \affP,
\end{align*}
where
\begin{align}
 \affP=\{\lam\in \dual{\affh};
\lam(K)=k,\ \bra \lam,\alpha_i\che\ket \in \Z
\ \text{for all }i=1,\dots,l \}.
\label{eq:affP}
\end{align}
Then
\begin{align}
 Pr_{\Z}^k=
\begin{cases}
 \{\lam\in \affP;
\bra \lam,\alpha_i\che \ket \geq 0
\ \text{for $i=1,\dots, l$,\ }\bra \lam,\theta\ket\leq
 p-h\che\}&\text{if }
(r\che
 ,q)=1,\\
\{\lam\in \affP; \bra \lam,\alpha_i\che \ket \geq 0
\ \text{for $i=1,\dots, l$,\ }\bra \lam,\theta_s\che \ket \leq
 p-h\}&\text{if }
(r\che ,q)=r\che.\end{cases}
\label{eq:Pr+}
\end{align}
We have  \cite{KacWak89}
\begin{align}
 Pr^k=\bigcup_{y\in \eW
\atop y(\wh{\Delta}(k\Lam_0)_+)\subset 
\wh{\Delta}^{re}_+}Pr_y^k,
\quad Pr_y^k:=y\circ Pr_{\Z}^k.
\end{align}

Note that
\begin{align*}
 Pr_{\Z}^k\cong \begin{cases}
	      \widehat{P}_{+}^{p-h\che}&\text{if }(r\che,q)=1,\\
{}^L\widehat{P}_{+}^{\vee,p-h}&\text{if }(r\che,q)=r\che,
	     \end{cases}
\end{align*}
where $\widehat{P}_{+}^m$
is the set of  level $m$  integral dominant weights of $\affg$,
and ${}^L\widehat{P}_{+}^{\vee,m}$ is the set of level $m$ integral dominant
coweights
of the affine Kac-Moody  algebra $\widehat{{}^L\fing}$
associated with the Langlands dual Lie algebra ${}^L\fing$.
Note  also that for 
$\lam\in Pr_{\Z}^k$
we have \begin{align}\label{eq:integral-roots}
 \widehat{\Delta}(\lam)
=\begin{cases}
			 \{\alpha+nq\delta; \alpha\in \Delta,\ n\in
			 \Z\}&\text{if }(q,r\che)=1,\\
\{\alpha+nq\delta;\alpha\in \Delta_{long}\}\sqcup
			 \{\alpha+\frac{nq}{r\che}\delta; \alpha\in\
			 \Delta_{short}, n\in \Z\}
&\text{if }(q,r\che)=r\che,
			\end{cases}
\end{align}
where $\Delta_{long}$ (resp.\ $\Delta_{short}$) is  the sets of
long roots
(resp.\ short roots) of $\fing$.
It follows that
\begin{align*}
 \widehat{W}(\lam)=\begin{cases}
					 W\ltimes q Q\che&\text{if
					 }(q,r\che)=1,\\
W\ltimes q Q&\text{if }(q,r\che)=r\che
					\end{cases}
\end{align*}
for $\lam\in Pr_{\Z}^k$.
In particular
\begin{align*}
 \widehat{W}(\lam)\cong \begin{cases}
					 \widehat{W}&\text{if
					 }(q,r\che)=1,\\
{}^L \widehat{W}&\text{if }(q,r\che)=r\che,
					\end{cases}
\end{align*}
for $\lam\in Pr^k$,
where ${}^L\widehat{W}$ is the Weyl group of $\widehat{{}^L\fing}$.

 \begin{Pro}[{\cite[1.5]{FKW92}}]\label{Pro:KW-nice}
Let $k$ be an admissible number for $\affg$ with denominator
$q$ 
as in Proposition \ref{Pro:admissible number}.
Suppose that  $(q,r\che)=1$,
and let
$\eW_+$ acts on
$Pr_{\Z}^k$ by 
$t_{\bar \Lam_j}w_j\mapsto t_{q\bar \Lam_j}w_j$,
 on $P\che/qQ\che$ be the identification 
$\eW_+=P\che/Q\che=qP\che/q Q\che$,
and on 
$Pr_{\Z}^k\times (P\che/q Q\che)$ diagonally.
We have  a bijection
\begin{align*}
 \begin{array}{ccc}
( Pr_{\Z}^k\times (P\che/q Q\che))
/\eW_+& \isomap & Pr^k,\\
{[(\lam,\mu)]}& \mapsto & t_{-\mu}w\circ \lam,
 \end{array}
\end{align*}
where $w$ is the  element of $W\ltimes q Q\che$ such that
$t_{-\mu}w(\widehat{\Delta}(k\Lam_0)_+)\subset 
\widehat{\Delta}^{re}_+$,
which exists uniquely.
In particular 
\begin{align*}
 |Pr^k|=|P\che/qP\che||\widehat{P}^{p-h\che}_{+}|=q^l |\widehat{P}^{p-h\che}_{+}|.
\end{align*}
 \end{Pro}

Let
${}^L\eW_+$
denote the 
group $\eW_+$ for $\widehat{{}^L\fing}$,
that is,
 the subgroup of 
the extended Weyl group
${}^L\eW$ of $\widehat{{}^L\fing}$
consisting of elements of length zero.
It acts on $Pr_{\Z}^k$ by the identification
$Pr_{\Z}^k\cong {}^L\widehat{P}_{+}^{\vee, p-h}$
if $r\che$ divides the denominator of $k$.
It acts also on the set
$P\che/q Q$ by the identification
${}^L\eW_+=P/Q=q P/q Q$,
and hence on
$Pr_{\Z}^k\times (P\che /q Q)$ diagonally.
The following assertion can be proved in the same manner as
Proposition \ref{Pro:KW-nice}.

\begin{Pro}\label{Pro:better-description-of-admissible-weights}
Let $k$ be an admissible number for $\affg$ with denominator
$q$ 
as in Proposition \ref{Pro:admissible number}.
Suppose that  $(q,r\che)=r\che$.
We have a bijection
\begin{align*}
 \begin{array}{ccc}
( Pr_{\Z}^k\times (P\che/q Q))
/{}^L\eW_+& \isomap & Pr^k,\\
{[(\lam,\mu)]}& \mapsto & t_{-\mu}w\circ \lam,
 \end{array}
\end{align*}
where
$w$ is the  element of $W\ltimes q Q$ such that
$t_{-\mu}w(\widehat{\Delta}(k\Lam_0)_+)\subset 
\widehat{\Delta}^{re}_+$,
which exists uniquely.
In particular
\begin{align*}
 |Pr^k|=
|P\che/q P||{}^L \widehat{P}_{+}^{\vee,p-h}|
=\frac{q^l|{}^L \widehat{P}_{+}^{\vee,p-h}|}
{(r\che)^{|\text{short simple roots of $\fing$}|}}.
\end{align*}
\end{Pro}

Now recall the following important result \cite{MalFre99} 
(see also \cite{FreMal97}).
\begin{Th}[{\cite{MalFre99}}]\label{Th:Frenkel-Malikov}
 Let $k$ be an admissible number and $\lam\in Pr_{\Z}^k$.
Then $L(\lam)$ is a module over $L(k\Lam_0)$.
\end{Th}
The following assertion follows immediately
from
Proposition \ref{Pro:zhu-vs-joseph}
and Theorem \ref{Th:Frenkel-Malikov}.
 \begin{Co}\label{Co:FM}
We have $\bar \lam\in \zeroset(I_k)$
for  $\lam\in Pr_{\Z}^k$.
 \end{Co}
\begin{proof}[Proof of the ``if part'' of Main Theorem]
Let $\mu\in Pr_{y}^k$.
Then there exist
$\lam\in Pr_{\Z}^k$
and $w\in \eW$
such that
$\mu=y\circ \lam$
and 
$y(\wh{\Delta}(\lam)_+)\subset \wh{\Delta}^{re}_+$.
(Note that
$\wh\Delta(\lam)_+=\Delta(k\Lam_0)_+$.)
The last condition is equivalent to that
if $\alpha\in \wh{\Delta}^{re}_+$
such that $\bra \lam+\rho,\alpha\che\ket\in \Z$
then $y(\alpha)\in \wh{\Delta}^{re}_+$.
Or equivalently,
$\bra \lam+\rho,\alpha\che\ket  \not \in \Z$ for all $\alpha\in 
\wh{\Delta}^{re}_+\cap y\inv (\wh{\Delta}^{re}_-)$.
Since
$\bar \lam$ belongs to $ \zeroset(I_k)$ by 
Theorem \ref{Th:Frenkel-Malikov},
so does $\bar \mu$ by
Lemma \ref{Lem:affin-Duflo-Joseph}.
By Proposition \ref{Pro:zhu-vs-joseph},
this completes the proof of the ``if part'' of Main Theorem.
\end{proof}

\section{Semi-infinite restriction functors}
\label{section:Semi-infinite restriction functors}
In this section we prove the ``only if'' part of Main Theorem by
reducing to the $\mf{sl}_2$-cases.
(Recall that
statement of Main Theorem has been proved for $\fing=\mf{sl}_2$ 
by  Adamovi\'c and Milas \cite{AdaMil95}.)

Let $k$ be an admissible number for $\affg$,
so that
\begin{align*}
 k+h\che=\frac{p}{q},\ p,q\in \N,\ (p,q)=1,\
p\geq\begin{cases}
      h\che&\text{if }(r\che,q)=1,\\ h&\text{if }(r\che,q)=r\che.
     \end{cases}
\end{align*}

Let
$\{e_i,h_i,f_i;
 i=1,\dots,l\}$ 
be  a set of Chevalley generators of $\fing$.
For $i=1,\dots,l$,
let 
$\finpi$ be 
the  minimal  parabolic subalgebra 
$\C f_i\+ \finh\+ \finn$ of $\fing$,
$\finli$ 
its Levi subalgebra,
$\finmi$ its nil-radical.
We have
\begin{align*}
\finpi
=\finli\+\finmi,\quad
\finli=\mf{sl}_2^{(i)}
\+\finh_i^\bot,
\end{align*}
where
$\mf{sl}_2^{(i)}$ is the copy of $\mf{sl}_2$ spanned by 
$e_i,h_i$ and $f_i$,
and $\finh_i^\bot$ is the orthogonal component of $\C h_i$ in $\finh$.
We have
$\finmi=\bigoplus_{\alpha\in \Delta(\finmi)}\C x_{\alpha}$,
where
$\Delta(\finmi)=\Delta_+\backslash \{\alpha_i\}$
and $x_{\alpha}$ is a root vector of $\fing$ of root $\alpha$.

Set
\begin{align*}
 &L\finmi=\finmi[t,t\inv]\subset \affg.
\end{align*}
We shall consider 
the semi-infinite $L\finmi$-cohomology
$H^{\semiinf+\bullet}(L\finmi, M)$
with coefficient in $M\in \BGG_k$.
This is defined by 
Feigin's standard complex  $(C^{\bullet}(L\finmi,M)
,d)
$ (\cite{Feu84});
$H^{\semiinf+\bullet}(L\finmi, M)=
H^i(C^{\bullet}(L\finmi,M),d)$,
$C^{\bullet}(L\finmi,M)=M\* \Lamsemi{\bullet}\finmi$,
where
$\Lamsemi{\bullet}\finmi$ is a vertex (super)algebra
generated by the odd field
$\psi_\alpha(z)=\sum_{n\in Z}\psi_{\alpha,n}z^{-n-1}$,
$\psi_{\alpha}^*(z)
=\sum_{n\in \Z}\psi_{\alpha,n}^* z^{-n}$,
$\alpha\in \Delta(\finmi)$,
satisfying the OPEs
\begin{align*}
 \psi_{\alpha}(z)\psi_{\beta}^*(w)\sim
 \frac{\delta_{\alpha,\beta}}{z-w},\quad 
 \psi_{\alpha}(z)\psi_{\beta}(w)\sim 
 \psi_{\alpha}^*(z)\psi_{\beta}^*(w)\sim 0.
\end{align*}
Differential $d$ is the coefficient of $z^{-1}$  of the field
\begin{align*}
 d(z)=\sum_{\alpha\in\Delta(\finmi)}x_{\alpha}(z)
 \psi_{\alpha}^*(z)-
 \frac{1}{2}\sum_{\alpha,\beta,\gamma
\in \Delta(\finmi)}c_{\alpha,\beta}^\gamma
:\psi_{\alpha}^*(z)\psi_{\beta}^*(z)\psi_{\gamma}(z):,
\end{align*}
where $c_{\alpha,\beta}^{\gamma}$ is the 
structure constant:
$[x_{\alpha},x_{\beta}]=\sum_{\gamma}c_{\alpha,\beta}^{\gamma}x_{\gamma}$
and we have omitted the tensor product symbol.

The space
$H^{\bullet}(L\finmi,\Vg{k})$ inherits the 
vertex (super)algebra structure from
$C^{\bullet}(L\finmi,\Vg{k})$.
It follows that the space
$H^{\semiinf+\bullet}(L\finmi, M)$,
$M\in \BGG_k$,
is a module over
the vertex algebra $H^{\semiinf+\bullet}(L\finmi, M)$.
Similarly, 
$H^{\bullet}(L\finmi,L(k\Lam_0))$ is
naturally a vertex algebra,
and if $M\in \BGG_k$ is a $L(k\Lam_0)$-module 
then 
$H^{\semiinf+\bullet}(L\finmi, M)$
is a module over
$H^{\semiinf+\bullet}(L\finmi, L(k\Lam_0))$.

Define the rational number $k_i$
by  the formula
 \begin{align}
 k_i+2=\frac{2}{(\alpha_i|\alpha_i)}(k+h\che)=
\begin{cases}
  \frac{p}{q}&\text{if }\alpha_i\text{ is a long root},\\
\frac{r\che p}{q}&\text{if }\alpha_i\text{ is a short root.}
 \end{cases}
\label{eq;ki}
\end{align}
Note that
$k_i$ is an admissible number for $\wh{\mf{sl}}_2$.

Let
\begin{align*}
 V^{k_i}(\finli)=V^{k_i}(\mf{sl}_2^{(i)})\* \pi,
\end{align*}
where $\pi$ is the Heisenberg vertex algebra associated with
$\finh_i^{\bot}[t,t\inv]\+ \C K\subset \affg$ at level $k+h\che$:
\begin{align*}
 \pi=U(\finh_i^\bot[t,t\inv]\+ \C K)\*_{U(\finh_i^{\bot}[t]\+ \C K)
}\C_{k+h\che},
\end{align*}
where  $\C_{k+h\che}$ is the one-dimensional representation 
of $\finh_i^{\bot}[t]\+ \C K$ on which $\finh_i^\bot[t]$ acts trivially and
$K$ acts as the scalar $k+h\che$.

For
$x\in \finli$,
set
\begin{align*}
 & \widehat x(z)=x(z)-\sum_{\beta,\gamma\in \Delta(\finmi)}
(x_{-\gamma}|[x,x_{\beta}]):\psi_{\beta}^*(z)\psi_{\gamma}(z):,
\end{align*}
where
$x_{-\gamma}$ is a root vector of root $-\gamma$ such that
$(x_{-\gamma}|x_{\gamma})=1$.
Then
$x(z)\mapsto \widehat x(z)$,
$x\in \finli$,
gives a 
vertex algebra homomorphism
$ V^{k_i}(\finli)
\ra C^{\bullet}(L\finmi, \Vg{k})$.
Since $\widehat{x}(z)$,
$x\in \finli$,
 commutes with the action of $d$,
it induces a vertex algebra homomorphism
\begin{align}
 V^{k_i}(\finli)\ra H^{\semiinf+0}(L\finmi, \Vg{k}),
\end{align}
see e.g.\ \cite{HosTsu91} for the details of the above facts.
Thus,
$H^{\semiinf+0}(L\finmi, M)$
, $M\in \BGG_k$,
is a module over $V^{k_i}(\finli)$.
In particular,
it is a module over $\widehat{\mf{sl}}_2^{(i)}$  of level $k_i$,
which
 belongs to the category $\BGG_{k_i}$ of 
$\widehat{\mf{sl}}_2^{(i)}$.

Denote by  $V_{k_i}(\finli)$  the simple quotient of
$V^{k_i}(\finli)$.
We have
\begin{align*}
 V_{k_i}(\finli)=L_{\wh{\mf{sl}}_2^{(i)}}(k_i\Lam_0)\*\pi,
\end{align*}
where
$L_{\widehat{\mf{sl}}_2^{(i)}}(\mu)$
denotes the irreducible
representation of $\wh{\mf{sl}}_2^{(i)}$ with highest weight $\mu$.
The quotient map
$\Vg{k}\ra L(k\Lam_0)$ induces the vertex algebra homomorphism
$H^{\semiinf+0}(L\finmi ,\Vg{k})\ra
H^{\semiinf+0}(L\finmi ,L(k\Lam_0))$.
Hence we have the vertex algebra homomorphism
\begin{align}
 V^{k_i}(\finli)
\ra H^{\semiinf+0}(L\finmi ,L(k\Lam_0)).
\label{eq:va-hom}
\end{align}

\begin{Th}[{\cite[Theorem 7.5]{A-BGG}}]\label{Th:reduction-original}
The vertex algebra homomorphism \eqref{eq:va-hom}
factors through the 
vertex algebra homomorphism
\begin{align*}
 V_{k_i}(\finli)
\ra
 H^{\semiinf+0}(L\finmi ,L(k\Lam_0))
\end{align*}
In particular, the fields
$\widehat{e}_i(z),
\widehat{h}_i(z),
\widehat{f}_i(z)$
generate
the admissible affine
 vertex algebra $L_{\wh{\mf{sl}}_2^{(i)}}(k_i \Lam_0)$
in
$ H^{\semiinf+0}(L\finmi ,L(k\Lam_0))$.
\end{Th}
\begin{proof}
  For reader's convenience we shall sketch the proof here
(see \cite{A-BGG} for the details).
For a weight $\mu\in \dual{\affh}$
of level $k$,
let
$W(\mu)$ be 
the Wakimoto module of $\affg$ with highest weight
 $\mu$ (\cite{FeuFre90,Fre05}).
We have
\begin{align}
 H^{\semiinf+i}(L\finmi , W(\mu))
\cong \begin{cases}
					     W_{\finli}(\mu)
&\text{for
					     }i=0,\\
0&\text{for }i\ne 0,
					    \end{cases}
\end{align}
where
$W_{\finli}(\mu)=W_{\wh{\mf{sl}}_2^{(i)}}(\mu^{(i)})\*
 \pi_{\mu|_{\finh_i^\bot}}$,
$\mu^{(i)}$ is a weight of $\wh{\mf{sl}}_2^{(i)}$ defined by
\begin{align*}
\mu^{(i)}=\mu|_{\C h_i}+ k_i \Lam_0,
\end{align*}
$W_{\wh{\mf{sl}}_2^{(i)}}(\mu^{(i)})$ is the Wakimoto module
of $\wh{\mf{sl}}_2^{(i)}$ with highest weight $\mu^{(i)}$
and $\pi_{\gamma}$ is the simple
$\pi$-module with highest weight $\gamma$.
In particular
$W(\mu)$ is  acyclic with respect to  the  functor
$ H^{\semiinf+i}(L\finmi , ?)$.
Hence one may compute
the cohomology
$ H^{\semiinf+i}(L\finmi , M)$
by using a resolution of $M$ in terms of Wakimoto modules.

Set $\lam=k\Lam_0$.
For the
admissible representation $L(\lam)$ there is a natural choice for
 such a resolution,
which is a two-sided analogue of the usual BGG resolution:
there exists a complex
\begin{align}
 C^{\bullet}: \cdots\ra  C^{-1}\overset{d_{-1}}{\ra} C^0\overset{d_0}{\ra} C^1\overset{d_1}{\ra}\cdots
\label{eq:BGG}
\end{align}
of $\affg$-modules
such that 
\begin{align*}
&C^i=\bigoplus\limits_{w\in \affW(\lam)\atop
\ell_{\lam}^{\semiinf}(w)
=i} W(w\circ \lam),\quad d_i=\sum\limits_{w,w'\in \affW(\lam)
\atop
\ell_{\lam}^{\semiinf}(w)=i, 
w\rhd w'}d_{w,w'},\\
\text{and }& H^i(C^{\bullet})=\begin{cases}
		   L(\lam)&\text{for }i=0,\\ 0&\text{for }i\ne 0,
		  \end{cases}
\end{align*}
where 
$\ell_{\lam}^{\semiinf}(w)$
denotes the semi-infinite length of $w\in \affW(\lam)$,
the symbol $\rhd$ denotes the covering
in the semi-infinite Bruhat order,
and $d_{w,w'}$
for $w\rhd w'$ is a non-trivial $\affg$-module homomorphism
$W(w\circ \lam)\ra W(w'\circ \lam)$,
which is unique up to multiplication  by a nonzero constant.
The existence of such a resolution was conjectured in \cite{FKW92},
and was proved in \cite{A-BGG} by applying
 a result \cite{Fie06} of
Fiebig
and a method \cite{Ark96} of  Arkhipov:
Fiebig's equivalence between different blocks  of categories 
$\BGG$ shows the existence of the 
usual BGG resolution of $L(\lam)$, 
and Arkhipov's method enables us to transform the usual BGG resolution
to the two-sided BGG resolution, see \cite{A-BGG} for the details.

Now the space
$H^{\semiinf+\bullet}(L\finmi ,L(\lam))$ 
is isomorphic to the
cohomology of the complex
$
H^{\semiinf+0}(L\finmi , C^{\bullet})
$
obtained from \eqref{eq:BGG}
by applying the functor
$H^{\semiinf+0}(L\finmi ,?)$,
which
 has the following form:
\begin{align*}
 \cdots \ra \bigoplus_{w\in \affW(k\Lam_0)
\atop
\ell^{\semiinf}_{\lam}(w)=-1}W_{\finli}(w\circ \lam)\overset{d_{-1}'}{\ra}
 \bigoplus_{w\in \affW(k\Lam_0)
\atop
\ell^{\semiinf}_{\lam}(w)=0}W_{\finli}(w\circ \lam)\overset{d_{0}'}{\ra}
 \bigoplus_{w\in \affW(k\Lam_0)
\atop
\ell^{\semiinf}_{\lam}(w)=1}W_{\finli}(w\circ \lam)\overset{d_{1}'}{\ra}\cdots,
\end{align*}
where
$ d_i'=\sum\limits_{w,w'\in \affW(k\Lam_0)\atop \ell^{\semiinf}(w)=i,\
 w\rhd w'}d_{w,w'}'$
and $d_{w,w'}'$ is the homomorphism $W_{\finli}(w\circ \lam)\ra
 W_{\finli}(w'\circ \lam)$ induced by $d_{w,w'}$.
Observe that in this  realization of $H^{\semiinf+\bullet}(L\finmi ,L(\lam))$ 
the vertex algebra homomorphism
\eqref{eq:va-hom} is obtained from
the vertex algebra embedding
\begin{align*}
 V^{k_i}(\finli)\hookrightarrow W_{\finli}(k\Lam_0)\subset H^{\semiinf+0}(L\finmi ,C^0)
\end{align*}
described in  \cite[Theorem 5.1]{Fre05}.
Therefore 
in order to prove the assertion
it remains to show that
the maximal proper submodule $N$ of $V^{k_i}(\finli)$ is contained in the
image of $d'_{-1}$.
Note that
$N$ is tensor product of the maximal proper submodule 
of
$V^{k_i}(\mf{sl}_2^{(i)})$ 
 and  $\pi$.

Set
\begin{align*}
\dot{\alpha}_0^{(i)}=\begin{cases}
		      -\alpha_i+q\delta&\text{if }\alpha_i\text{ is a
		      long root or } r\che\not|q,\\
-\alpha_i+\frac{q}{r\che}\delta&\text{if }\alpha_i\text{ is a
		      short root and } r\che|q,
		     \end{cases}
\end{align*}
and put
$\dot{s}_0^{(i)}=s_{\dot{\alpha}_0^{(i)}}$.
Then
$\dot{\alpha}_0^{(i)}\in \Delta(k\Lam_0) $,
$\ell^{\semiinf}_{\lam}(\dot{s}_0^{(i)}
)=-1$,
$\dot{s}_0^{(i)}\rhd 1$,
and the maximal submodule of
$V^{k_i}(\mf{sl}_2)$ 
is generated by a singular vector of weight
$(\dot{s}_0^{(i)}
\circ k\Lam_0)^{(i)}$ (\cite{KacWak88}).
One sees that
the image of the highest weight vector
of $W_{\finli}(\dot{s}_0^{(i)}\circ k\Lam_0)$ in $W_{\finli}(k\Lam_0)$
by the homomorphism
$d_{\dot{s}_0^{(i)},1}'$ is nonzero and generates the
maximal submodule of $V^{k_i}(\finli)\subset W_{\finli}(k\Lam_0)$.
This completes the proof.
\end{proof}

\begin{Th}[{\cite[Theorem 7.6]{A-BGG}}]\label{Th:reduction}
If $M$ is a module over $\Vs{k}$,
the space $H^{\semiinf+r}(L\finmi , M)$,
$r\in \Z$,
 is a direct sum of
irreducible admissible representations of $\wh{\mf{sl}}_2^{(i)}$ of level
$k_i$.
\end{Th}
\begin{proof}
By Theorem
\ref{Th:reduction-original},
$H^{\semiinf+\bullet}(L\finmi ,M)$
is a module over $L_{\widehat{\mf{sl}}_2^{(i)}}(k_i \Lam_0)
\subset V_{k_i}(\finli)$,
which belongs to the category 
$\BGG$ of $\widehat{\mf{sl}}_2^{(i)}$ of level $k_i$.
Hence the assertion follows immediately from
 \cite{AdaMil95}.
\end{proof}

Let
\begin{align*}
 L^{tot}(z)=\sum_{n\in \Z}L^{tot}_nz^{-n-2}
:=L(z)+:\sum_{\alpha\in \Delta(\finmi)}:
\psi_{\alpha}(z)\partial_z \psi_{\alpha}^*(z):.
\end{align*}
Then
$[L^{tot}_0,xt^n]=-n xt^n$
for $x\in \fing$,
$[L^{tot}_0, \psi_{\alpha,n}]=-n\psi_{\alpha,n}$,
$[L^{tot}_0, \psi_{\alpha,n}^*]=-n\psi_{\alpha,n}^*$.
It follows that
$L^{tot}$ commutes with $d$, and thus, acts on the space
$H^{\semiinf+0}(L\finmi, M)$.

For $M\in \BGG_k$,
let
$C^{\bullet}(L\finmi,M)_{d}\subset C^{\bullet}(L\finmi,M)$,
$H^{\semiinf+\bullet}(L\finmi, M)_d\subset
H^{\semiinf+\bullet}(L\finmi, M)$,
be
the generalized $L_0^{tot}$-eigenspaces of eigenvalue $d$.
Then
\begin{align*}
&C^{\bullet}(L\finmi,M)=\bigoplus_{d\in \C}C^{\bullet}(L\finmi,M)_{d}\\
 &H^{\semiinf+\bullet}(L\finmi, M)
=\bigoplus_{d\in \C}H^{\semiinf+
\bullet}(L\finmi, M)_{d}.
\end{align*}
Note that
$C^{\bullet}(L\finmi,M)_{d}$ is a subcomplex
of $C^{\bullet}(L\finmi,M)$
and 
$H^{\semiinf+
\bullet}(L\finmi, M)_{d}$ is the cohomology of the subcomplex
$C^{\bullet}(L\finmi, M)_{d}$.

\begin{Lem}\label{Lem:top-part-of-cohomology}
Let $M$ be a positively graded $\affg$-module in the category $\BGG_k$, 
and
let $d_{top}$ be the $L_0$-eigenvalue on $M_{top}$.
Then
$H^{\semiinf+\bullet}(L\finmi,M)=\bigoplus\limits_{n\in \Z_{\geq 0}}
H^{\semiinf+\bullet}(L\finmi,M)_{d_{top}+n}$
and
\begin{align*}
 H^{\semiinf+r}(L\finmi, M)_{d_{top}}\cong 
\begin{cases}
 H^{r}(\finmi,M_{top})&\text{if }r\geq 0,\\
0&\text{otherwise,}
\end{cases}
\end{align*}
where
$H^{\bullet}(\finmi,N)$
denotes the (usual)
Lie algebra $\finmi$-cohomology with coefficient in an $\finmi$-module $N$.
\end{Lem}
\begin{proof}
The first assertion is easy to see.
 The second assertion follows by observing 
that $C^{\bullet}(L\finmi, M)_{d_{top}}=M_{top}\otimes
\bw{\bullet }(\finmi)$ and the restriction of the
differential  to
$C^{\bullet}(L\finmi, M)_{d_{top}}$ coincides with the differential of
 the Chevalley complex for calculating 
$H^{\bullet}(\finmi,M_{top})$.
\end{proof}

\begin{Lem}\label{Lem:sl_2-adimissible-and-so}
Let $k$ be an admissible number
and 
suppose that $L(\lam)$ is a module over  $\Vs{k}$.
Then
$\lam^{(i)}=\lam|_{\C h_i}\+ k_i \Lam_0$
is  an admissible weight for $\wh{\mf{sl}}_2$
for all $i=1,\dots, l$.
\end{Lem}
\begin{proof}
Let $d_{top}$ be the top weight of
$L(\lam)$.
By Lemma \ref{Lem:top-part-of-cohomology},
$H^{\semiinf+0}(L\finmi, L(\lam))
=\bigoplus_{d\geq  
d_{top}}H^{\semiinf+0}(L\finmi, L(\lam))_{d}$
and
\begin{align*}
H^{\semiinf+0}(L\finmi, L(\lam))_{d_{top}}=
H^{0}(\finmi, \sL^{\fing}_{\bar \lam}).
\end{align*}
It follows that the image $[v_{\lam}]$
of the highest weight vector
$v_{\lam}$
of $L(\lam)$ 
is a non-zero singular vector of $\wh{\mf{sl}}_2^{(i)}$
of weight
$\lam^{(i)}$
in $H^{\semiinf+0}(L\finmi, L(\lam))$.
Hence   $[v_{\lam}]$ generates 
a (nonzero) highest weight $\wh{\mf{sl}}_2^{(i)}$-submodule
with highest weight $\lam^{(i)}$.
But according to    Theorem  \ref{Th:reduction}
such a module must be admissible.
This completes the proof.
\end{proof}
\begin{Pro}\label{Pro:integral-Weyl-groups-of-a-module}
 Suppose that $L(\lam)$ is a $\Vs{k}$-module.
Then $\wh{\Delta}(\lam)\cong \wh{\Delta}(k\Lam_0)$.
\end{Pro}
\begin{proof}
 By Lemma \ref{Lem:sl_2-adimissible-and-so},
 $\bra \lam+\rho,\alpha_i\che\ket \in
 \frac{2}{(\alpha_i|\alpha_i)q}\Z$
for all $i=1,\dotsm l$.
It follows that
there exists
 $n_i\in \Z$
for each  $i=1,\dots, l$  
such that
$\alpha_i+n_i\delta\in \wh{\Delta}(\lam)$.
Hence
\begin{align}
\text{there exists $n_{\alpha}\in \Z$
for each $\alpha\in \Delta$
such that
 } \alpha+n_{\alpha}\delta\in \wh{\Delta}(\lam).
\label{eq:missing}
\end{align}
This implies that  $\Q\wh{\Delta}(\lam)=\Q \wh{\Delta}^{re}$.
Therefore  \cite{KacWak89}
$\wh{\Delta}(\lam)$ is isomorphic to the integral root system
of some admissible weight of level $k$.
But according to
 the classification of admissible weights 
by Kac and Wakimoto \cite{KacWak89},
the integral roots of admissible weights satisfying
 the property (\ref{eq:missing})
must be  isomorphic to $\wh{\Delta}(k\Lam_0)$.
\end{proof}

\begin{Pro}\label{Pro:the case-G-integrable-only-if}
 Let $k$ be an admissible number and 
$\lam\in \affP$ (see \eqref{eq:affP}).
Suppose that
$L(\lam)$ is a module over $\Vs{k}$.
Then $L(\lam)$ is admissible.
\end{Pro}
\begin{proof}
 First, it follows that 
\begin{align*}
 \text{$\bra \lam,
\alpha_i\che\ket \geq 0$ for $i=1,\dots,l$
}
\end{align*}from Lemma \ref{Lem:sl_2-adimissible-and-so}.
Therefore
the $\fing$-submodule $L^{\fing}_{\bar \lam}$
generated by the highest weight vector
of $L(\lam)$ is finite-dimensional.
By \cite{Kos61}
we have
\begin{align*}
 H^r(\finmi, L^{\fing}_{\bar \lam}))\cong \bigoplus_{w\in \finW^{(i)}\atop
\ell(w)=r} L^{\finli}_{w\circ \bar \lam}
\end{align*}
for $r\geq 0$
as $\finli$-modules,
where
$L^{\finli}_{\bar \mu}$ is the irreducible highest weight
 representation
of $\finli$ with highest weight $\bar \mu$
and
\begin{align*}
 \finW^{(i)}=\{w\in \finW; w\inv (\alpha_i)\in \Delta_+
 \}.
\end{align*}
Therefore
 \begin{align*}
  H^{\semiinf+r}(L\finmi, L(\lam))_{d_{top}}
\cong \bigoplus_{w\in \finW^{(i)}\atop
\ell(w)=r} L^{\finli}_{w\circ \bar \lam}
 \end{align*}
by Lemma \ref{Lem:top-part-of-cohomology}.
Since
each $\finli$-highest weight vector of 
$  H^{\semiinf+r}(L\finmi, L(\lam))_{d_{top}}$
generates 
a highest weight representation
of $\wh{\mf{sl}}_2^{(i)}$,
the weights
$(w\circ \lam)^{(i)}$
 must be admissible
by Theorem \ref{Th:reduction},
that is,
\begin{align}
(w\circ \lam)^{(i)}
\in Pr_{\Z}^{(i),k_i}
\text{ for all $w\in W^{(i)}$,
$i=1,\dots, l$.
}\label{eq:condidion}
\end{align}
where 
$Pr_{\Z}^{(i),k_i}
$ denotes 
the set $Pr^{k_i}_{\Z}$ for $\widehat{\mf{sl}}_2^{(i)}$.

Now first consider the case that
$(r\che, q)=1$.
It remains to show 
that $\bra \lam,\theta\ket\leq p-h^{\vee}$, see \eqref{eq:Pr+}.
Let $\alpha_i$ be any  simple long root of $\fing$,
so that
 there exists $w\in \finW$
such that $\theta=w\inv (\alpha_i)$.
Since
$(w\circ \lam)^{(i)}\in Pr_{\Z}^{(i),k_i}$
and  $k_i+2=p/q$,
we have
\begin{align*}
 p-2\geq \bra w\circ  \lam,\alpha_i\che\ket=
\bra w(\lam+\rho),\alpha_i\che\ket-1
=\bra \lam+\rho, \theta\ket -1
=\bra \lam,\theta\ket +h\che -2.
\end{align*}
We have shown that
 $\lam\in Pr_{\Z}^k$.

Next consider the case that  $(r\che,q)=r\che$.
We need to
show that
$\bra \lam,\theta_s\che\ket\leq  p-h$.
Let $\alpha_i$ be any  simple short root of $\fing$.
Then there exists $w\in \finW$ such that $\theta_s=w\inv(\alpha_i)$.
Since
$(w\circ \lam)^{(i)}\in Pr_{\Z}^{(i),k_i}$
and  $k_i+2=r\che p/q=p/(q/r\che)$,
we have
\begin{align*}
 p-2\geq \bra w\circ  \lam,\alpha_i\che\ket=
\bra w(\lam+\rho),\alpha_i\che\ket-1
=\bra \lam+\rho, \theta_s\che\ket -1
=\bra \lam,\theta_s\che\ket +h -2.
\end{align*}
Hence
 $\lam\in Pr_{\Z}^k$ as required.
\end{proof}
 \begin{proof}[Proof of the ``only if'' part of Main Theorem]
Suppose that $L(\lam)$ is a module over  $\Vs{k}$.
By Proposition \ref{Pro:integral-Weyl-groups-of-a-module},
$\wh{\Delta}(\lam)\cong \wh{\Delta}(k\Lam_0)$.
Therefore 
by \cite[Lemma 2.1]{KacWak89} there exist 
$y\in \eW$
such that 
$\widehat{\Delta}(\lam)_+=y(\widehat{\Delta}(k\Lam_0)_+)
$.
Let $\mu=y\inv \circ \lam$.
Then 
\begin{align*}
\wh\Delta(\mu)_+=\wh\Delta(k\Lam_0)_+=y\inv (\widehat{\Delta}(\lam)_+).
\end{align*}
This gives 
on the one hand
that 
$\mu\in \affP$
and
on the other hand
$\bra \mu+\rho,\alpha\che\ket\not\in \Z$
for $\alpha\in \widehat{\Delta}^{re}_+\cap y(\widehat{\Delta}^{re}_-)$.
The latter condition implies that
$L(\mu)$ is also  a module over
$\Vs{k}$ by Lemma \ref{Lem:affin-Duflo-Joseph}.
But then Proposition \ref{Pro:the case-G-integrable-only-if}
implies that 
$\mu$ must be an admissible weight
 since $\mu\in \affP$.
Therefore
$\lam$ is also an admissible weight.
This completes the proof of the ``only if'' part of Main Theorem.
 \end{proof}
The last assertion of Main Theorem
now follows immediately 
from the fact that
$\on{Ext}_{\BGG_k}^1(L(\lam),L(\mu))=0$
 for any 
admissible 
weights $\lam,\mu$
of level $k$
since 
they are regular dominant weights.
In fact it is known by \cite[Theorem 0.2]{GorKac0905}
that  $\on{Ext}_{\affg}^1(L(\lam),L(\mu))=0$ for any 
admissible weights
$\lam,\mu$.
(Note that it is obvious 
that
$\on{Ext}_{\BGG_k}^1(L(\lam),L(\lam))=0$,
but it is highly  non-trivial that
$\on{Ext}_{\affg}^1(L(\lam),L(\lam))=0$.)
Hence we have the following slightly more general assertion
than 
the last assertion of Main Theorem.
\begin{Pro}
 Let
$M$ be a finitely generated $L(k\Lam_0)$-module
on which 
$\affn_+$ locally nilpotently and 
$\affh$ acts locally finitely.
Then $M$ is a direct sum of 
admissible representations of $\affg$
whose integral Weyl groups are isomorphic to that of $L(k\Lam_0)$.
\end{Pro}
\qed

\appendix
\section{Generalized semi-infinite Borel-Weil-Bott Theorem
for admissible representations}
\label{section:Generalized-Borel-Weil-Bott}

Let $\finp$
 be a parabolic subalgebra 
of $\fing$
containing
$\finb_-$,
and 
let
$\finp=\finl\+\finr_-$
be the direct sum decomposition of $\finp$
with the Levi subalgebra $\finl$  containing $\finh$
and the nilpotent radical 
 $\finm_-$.
Denote by  $\finm\subset \finn$ the opposite algebra of $\finm_-$,
so that
$\fing=\finp\+\finm$.
Let
\begin{align*}
 \finl=\finl_0\+\bigoplus_{i=1}^s \finl_i
\end{align*}
be the decomposition of $\finl$ into 
direct sum of 
simple Lie subalgebras $\finl_i$,
$i=1,\dots, s$, and its center $\finl_0$ of $\finl$.
Let $\finh_i=\finl\cap \finh$,
the Cartan subalgebra of $\finl_i$,
and 
denote by
$\sDelta_i\subset \sDelta$ the subroot system of $\fing$ corresponding
 to $\finl_i$,
$\Delta_{i,+}=\Delta_i\cap \Delta_+$,
$\sPi_i=\sPi\cap \sroots_i$.
Let  $h_i\che$ be the dual Coxeter number of $\finl_i$ (with a convention
that $h_0\che=0$),
  $\theta_i$ the highest  root of $\sDelta_i$,
$\theta_{i,s}$ 
the highest short roof of $\sDelta_i$.

Let 
$
\affl_i=\finl_i[t,t\inv]\+ \C K\subset \affg
$
for 
$i=0,1,\dots,s$.
Set
\begin{align*}
 K_i=\frac{2}{(\theta_i|\theta_i)}K,
\end{align*}
and we consider $K_i$ as an element of $\affl_i$.
Thus,
\begin{align*}
\affl_i=\finl_i[t,t\inv]\+ \C K_i.
\end{align*}
Then
$\affh_i=\finh_i\+ \C K_i$ is a Cartan subalgebra
of $\affl_i$.
Set  $\affh^*_i=\finh_i^*\+ \C \Lam_{0,i}\subset \affh^*$,
where $\Lam_{0,i}=\frac{(\theta_i|\theta_i)}{2}\Lam_0$,
and we regard $\affh^*_i$ as the dual of $\affh_i$.

Set 
\begin{align*}
 \widehat{ \Delta}_{\finl_i}^{re}
=\{\alpha+n\delta\in
 \widehat{\Delta}^{re};
\alpha_i\in \Delta_i\}
\end{align*}
Then $\widehat{ \Delta}_{\finl_i}^{re}$ is a subroot system of 
$\wh{\Delta}^{re}$.
We regard 
$\widehat{ \Delta}_{\finl_i}^{re}$
as a real root system of 
$\affl_i$.
In particular subgroup 
of $\affW$
generated $s_{\alpha}$, $\alpha\in \widehat{\Delta}^{re}_+$,
is identified with the Weyl group
$\affW_i$ of $\affl_i$.

For $\lam\in \affh^*$
define
$\lam_i\in \affh_i^*$,
$i=0,1,\dots, s$, by
\begin{align*}
\bra  \lam_i+\widehat{\rho}_i,\alpha\che +n K_i\ket=
\bra \lam +\widehat{\rho},\alpha\che
 +\frac{2n}{(\theta_i|\theta_i)}K\ket
\quad\text{for }\alpha\che\in \finh_i,
\end{align*}
where $\widehat{\rho}_i=\rho_i+h_i\che \Lam_{0,i}$,
$\rho_i=\sum_{\alpha\in \Delta_{i,+}}\alpha/2$.
Observe that  if 
$\lam$ is an admissible weight of  $\affg$ level $k$, 
then $\lam_i$ is an admissible weight of $\affl_i$
of level $k_i$ for $i=1,2,\dots, s$,
where 
\begin{align*}
 k_i+h_i\che=\frac{2}{(\theta_i|\theta_i)}(k+h\che).
\end{align*}

Define Lie algebras
\begin{align*}
 \affl=\bigoplus_{i=0}^s \affl_i,
\quad \afft=\bigoplus_{i=0}^s\affh_i\subset \affl.
\end{align*}
For $\lam\in \dual{\affh}$,
set
\begin{align*}
 \lam_{\finl}=\sum_{i=0}^s \lam_i\in \afft^*,
\end{align*}
and let
$L_{\finl}(\lam_{\finl})$ be the
irreducible highest weight representation of $\affl$ with highest weight $\lam_{\finl}$.
Then
$L_{\finl}(\lam_{\finl})=\bigotimes_{i=0}^s L_{\finl_i}(\lam_i)$,
where $L_{\finl_i}(\lam_i)$ is the irreducible highest weight
representation
of $\affl_i$ with highest weight $\lam_i$.
The weight 
$\lam_{\finl}\in \dual{\afft}$ is called admissible
if $\lam_i$ is admissible for all $i\ne 0$,
and if this is the case
the $\affl$-module $L_{\finl}(\lam)$ is called admissible.

Now 
let $k$ be an admissible  number for $\affg$, and let $\lam\in Pr_{\Z}^k$.
Set
\begin{align*}
 \widehat{\Delta}_{\finl}(\lam)=\bigcup_{i=1}^s
 \widehat{\Delta}_{\finl_i}(\lam),
\quad
\widehat{ \Delta}_{\finl_i}(\lam)=
\widehat{\Delta}^{re}_{\finl_i}\cap  \widehat{\Delta}(\lam),
\end{align*}
and let
$\affW_{\finl}(\lam)$ be the subgroup
of $\affW(\lam)$ generated by
$s_{\alpha}$,
$\alpha\in \widehat{\Delta}_{\finl}(\lam)$.

Define
\begin{align*}
 \widehat{W}^{\finl}(\lam)=\{w\in \affW(\lam); w^{-1}(\widehat{\Delta}(\lam)_+)
\subset \widehat{\Delta}^{re}_+\}.
\end{align*}
Then 
 the multiplication map
$\widehat{W}_{\finl}(\lam)\times \widehat{W}^{\finl}(\lam)\ra \affW(\lam)$,
$(u,v)\mapsto uv$,
is a bijection. Moreover we have
\begin{align*}
 \ell^{\semiinf}_{\lam}(uv)=\ell^{\semiinf}_{\lam}(u)+\ell^{\semiinf}_{\lam}(v)
\end{align*}
for $u\in \affW_{\finl}(\lam)$, $v\in \affW^{\finl}(\lam)$
(\cite{Pet}, see \cite[Theorem 3.3]{A-BGG}).

 \begin{Lem}\label{Lem:easy}
\begin{enumerate}
 \item
      For $w\in \affW(\lam)$,
$(w\circ \lam)_{\finl}$ 
is admissible  if and only if 
$w\in \affW^{\finl}(\lam)$.
\item For $w,w'\in \affW^{\finl}(\lam)$,
$(w\circ \lam)_{\finl}=(w'\circ \lam)_{\finl}$ if and only if $w=w'$.
\end{enumerate}
 \end{Lem}

Let $L\finm=\finm[t,t\inv]\subset \affg$.
The semi-infnite $L\finm$-cohomology 
$H^{\semiinf+0}(L\finm,L(k\Lam_0))$
 is naturally an $\affl$-module.
In fact, 
as in the case that
$\finp$ is a minimal parabolic subalgebra (see Theorem
\ref{Th:reduction-original}),
there is an injective vertex algebra homomorphism
\begin{align}
L_{\finl}((k\Lam_0)_{\finl})
=
\bigotimes_{i=0}^s L_{\finl_i}(k_i \Lam_{i,0})
\hookrightarrow 
H^{\semiinf+0}(L\finm,L(k\Lam_0))
\end{align}
provided that $k$ is admissible (\cite[Theorem 7.5]{A-BGG}).
Since $L(\lam)$ is a module over 
$L(k\Lam_0)$,
$H^{\semiinf+p}(L\finm,L(\lam))$ is a module over the vertex algebra 
$H^{\semiinf+0}(L\finm,L(k\Lam_0))$,
and hence, is a module over 
$L_{\finl}((k\Lam_0)_{\finl})$.
Since each $k_i$ is an admissible number for $\affl_i$ for all $i\ne 0$,
it follows from Main Theorem that
$H^{\semiinf+p}(L\finm,L(\lam))$
 is a direct sum of admissible
representations of
$\affl$ for each  $p\in \Z$.

The following assertion 
can be proved in the same manner as
in \cite[Theorem 7.7]{A-BGG}.
\begin{Th}\label{Th:gen-Borel-Weil}
 Let $k$ be an admissible number for $\affg$,
$\lam\in Pr_{\Z}^k$.
For each $p\in \Z$ we have the  $\affl$-module isomorphism
\begin{align*}
 H^{\semiinf+p}(L\finm,L(\lam))
\cong \bigoplus_{w\in \affW^{\finl}(\lam)
\atop \ell^{\semiinf}_{\lam}(w)=p}L_{\finl}((w\circ \lam)_\finl).
\end{align*}
\end{Th}

 \begin{Rem}
Although an admissible affine vertex algebra has only  finitely many
  isomorphism classes of simple
  modules in the category $\BGG$
the  sum in Theorem \ref{Th:gen-Borel-Weil}
  is infinite since $\affl_0$ is a Heisenberg  algebra
which has infinitely many isomorphism classes of simple modules
in the category $\BGG$.

Theorem \ref{Th:gen-Borel-Weil}
should be regarded as an analogue of Kostant's generalized
  Borel-Weil-Bott Theorem \cite{Kos61}  appeared in the proof of 
Proposition \ref{Pro:the case-G-integrable-only-if}.
It has been proved in \cite{HosTsu91} in the case that 
$L(\lam)$ is an integrable representation of $\affg$
using the fact that integrable representations are unitarizable,
following the original idea of Kostant's proof.
Their proof does not apply to our case
since
 non-integrable admissible representations are not
  unitarizable.

We also note that in the above formula 
it was essential that  $L(\lam)$ 
is an admissible representations in order to
 to apply Main Theorem for $\affl_i$, $i\ne 0$.
Otherwise,
the 
complete reduciblility does not necessarily hold and the  
cohomology 
$H^{\semiinf+p}(L\finm,L(\lam))$ may be very complicated.

 \end{Rem}


\begin{thebibliography}{FKW}

\bibitem[Ada]{Ada94}
Dra{\v{z}}en Adamovi{\'c}.
\newblock Some rational vertex algebras.
\newblock {\em Glas. Mat. Ser. III}, 29(49)(1):25--40, 1994.

\bibitem[AL]{AxtLee11}
Jonathan~D. Axtell and Kyu-Hwan Lee.
\newblock Vertex operator algebras associated to type {$G$} affine {L}ie
  algebras.
\newblock {\em J. Algebra}, 337:195--223, 2011.

\bibitem[ALY]{ALY}
Tomoyuki Arakawa, Ching~Hung Lam, and Hiromichi Yamada.
\newblock Zhu's algebra, {$C_2$}-algebra and {$C_2$}-cofiniteness of
  parafermion vertex operator algebras.
\newblock {\em Adv. Math.}, 264:261--295, 2014.

\bibitem[AM]{AdaMil95}
Dra{\v{z}}en Adamovi{\'c} and Antun Milas.
\newblock Vertex operator algebras associated to modular invariant
  representations for {$A_{1}^{(1)}$}.
\newblock {\em Math. Res. Lett.}, 2(5):563--575, 1995.



\bibitem[A1]{Ara12}
Tomoyuki Arakawa.
\newblock A remark on the {$C_2$} cofiniteness condition on vertex algebras.
\newblock {\em Math. Z.}, 270(1-2):559--575, 2012.

\bibitem[A2]{Ara09b}
Tomoyuki Arakawa.
\newblock Associated varieties of modules over {K}ac-{M}oody algebras and
  {$C_2$}-cofiniteness of {W}-algebras.
\newblock arXiv:1004.1554[math.QA].

\bibitem[A3]{A-BGG}
T.~Arakawa.
\newblock Two-sided {BGG} resolution of admissible representations.
\newblock {\em Represent. Theory}, 19(3):183--222, 2014.

\bibitem[A4]{A2012Dec}
Tomoyuki Arakawa.
\newblock Rationality of {W}-algebras; principal nilpotent cases.
\newblock arXiv:1211.7124[math.QA].

\bibitem[Ark]{Ark96}
Sergey Arkhipov.
\newblock A new construction of the semi-infinite {BGG} resolution.
\newblock math.QA/9605043.

\bibitem[DLM]{DonLiMas97}
Chongying Dong, Haisheng Li, and Geoffrey Mason.
\newblock Vertex operator algebras associated to admissible representations of
  {$\widehat{\rm sl}\sb 2$}.
\newblock {\em Comm. Math. Phys.}, 184(1):65--93, 1997.

\bibitem[Duf]{Duf77}
Michel Duflo.
\newblock Sur la classification des id\'eaux primitifs dans l'alg\`ebre
  enveloppante d'une alg\`ebre de {L}ie semi-simple.
\newblock {\em Ann. of Math. (2)}, 105(1):107--120, 1977.

\bibitem[FBZ]{FreBen04}
Edward Frenkel and David Ben-Zvi.
\newblock {\em Vertex algebras and algebraic curves}, volume~88 of {\em
  Mathematical Surveys and Monographs}.
\newblock American Mathematical Society, Providence, RI, second edition, 2004.

\bibitem[Fe{\u\i}]{Feu84}
B.~L. Fe{\u\i}gin.
\newblock Semi-infinite homology of {L}ie, {K}ac-{M}oody and {V}irasoro
  algebras.
\newblock {\em Uspekhi Mat. Nauk}, 39(2(236)):195--196, 1984.

\bibitem[FF]{FeuFre90}
Boris~L. Fe{\u\i}gin and Edward~V. Frenkel.
\newblock Affine {K}ac-{M}oody algebras and semi-infinite flag manifolds.
\newblock {\em Comm. Math. Phys.}, 128(1):161--189, 1990.

\bibitem[Fie]{Fie06}
Peter Fiebig.
\newblock The combinatorics of category {$\mathcal{O}$} over symmetrizable
  {K}ac-{M}oody algebras.
\newblock {\em Transform. Groups}, 11(1):29--49, 2006.

\bibitem[FKW]{FKW92}
Edward Frenkel, Victor Kac, and Minoru Wakimoto.
\newblock Characters and fusion rules for {$W$}-algebras via quantized
  {D}rinfel\cprime d-{S}okolov reduction.
\newblock {\em Comm. Math. Phys.}, 147(2):295--328, 1992.

\bibitem[FeM]{FeiMal97}
Boris Feigin and Fyodor Malikov.
\newblock Modular functor and representation theory of {$\widehat{\rm sl}_2$}
  at a rational level.
\newblock In {\em Operads: Proceedings of Renaissance Conferences (Hartford,
  CT/Luminy, 1995)}, volume 202 of {\em Contemp. Math.}, pages 357--405,
  Providence, RI, 1997. Amer. Math. Soc.

\bibitem[FrM]{FreMal97}
Igor Frenkel and Fyodor Malikov.
\newblock {K}azhdan-{L}usztig tensoring and {H}arish-{C}handra categories.	  
\newblock arXiv:q-alg/9703010.

\bibitem[Fre]{Fre05}
Edward Frenkel.
\newblock Wakimoto modules, opers and the center at the critical level.
\newblock {\em Adv. Math.}, 195(2):297--404, 2005.

\bibitem[FZ]{FreZhu92}
Igor~B. Frenkel and Yongchang Zhu.
\newblock Vertex operator algebras associated to representations of affine and
  {V}irasoro algebras.
\newblock {\em Duke Math. J.}, 66(1):123--168, 1992.

\bibitem[GK]{GorKac0905}
Maria Gorelik and Victor Kac.
\newblock On complete reducibility for infinite-dimensional {L}ie algebras.
\newblock {\em Adv. Math.}, 226(2):1911--1972, 2011.

\bibitem[HT]{HosTsu91}
Shinobu Hosono and Akihiro Tsuchiya.
\newblock Lie algebra cohomology and {$N=2$} {SCFT} based on the {GKO}
  construction.
\newblock {\em Comm. Math. Phys.}, 136(3):451--486, 1991.

\bibitem[Jos]{Jos77}
A.~Joseph.
\newblock A characteristic variety for the primitive spectrum of a semisimple
  {L}ie algebra.
\newblock In {\em Non-commutative harmonic analysis ({A}ctes {C}olloq.,
  {M}arseille-{L}uminy, 1976)}, pages 102--118. Lecture Notes in Math., Vol.
  587. Springer, Berlin, 1977.

\bibitem[Kac]{Kac98}
Victor Kac.
\newblock {\em Vertex algebras for beginners}, volume~10 of {\em University
  Lecture Series}.
\newblock American Mathematical Society, Providence, RI, second edition, 1998.

\bibitem[Kos]{Kos61}
Bertram Kostant.
\newblock Lie algebra cohomology and the generalized {B}orel-{W}eil theorem.
\newblock {\em Ann. of Math. (2)}, 74:329--387, 1961.

\bibitem[KW1]{KacWak88}
Victor~G. Kac and Minoru Wakimoto.
\newblock Modular invariant representations of infinite-dimensional {L}ie
  algebras and superalgebras.
\newblock {\em Proc. Nat. Acad. Sci. U.S.A.}, 85(14):4956--4960, 1988.

\bibitem[KW2]{KacWak89}
V.~G. Kac and M.~Wakimoto.
\newblock Classification of modular invariant representations of affine
  algebras.
\newblock In {\em Infinite-dimensional Lie algebras and groups
  (Luminy-Marseille, 1988)}, volume~7 of {\em Adv. Ser. Math. Phys.}, pages
  138--177. World Sci. Publ., Teaneck, NJ, 1989.

\bibitem[KW3]{KacWak08}
Victor~G. Kac and Minoru Wakimoto.
\newblock On rationality of {$W$}-algebras.
\newblock {\em Transform. Groups}, 13(3-4):671--713, 2008.

\bibitem[Li]{Li97}
Haisheng Li.
\newblock The physics superselection principle in vertex operator algebra
  theory.
\newblock {\em J. Algebra}, 196(2):436--457, 1997.

\bibitem[MF]{MalFre99}
F.~G. Malikov and I.~B. Frenkel{\cprime}.
\newblock Annihilating ideals and tilting functors.
\newblock {\em Funktsional. Anal. i Prilozhen.}, 33(2):31--42, 95, 1999.

\bibitem[Per1]{Per07}
Ozren Per{\v{s}}e.
\newblock Vertex operator algebras associated to type {$B$} affine {L}ie
  algebras on admissible half-integer levels.
\newblock {\em J. Algebra}, 307(1):215--248, 2007.

\bibitem[Per2]{Per08}
Ozren Per{\v{s}}e.
\newblock Vertex operator algebras associated to certain admissible modules for
  affine {L}ie algebras of type {$A$}.
\newblock {\em Glas. Mat. Ser. III}, 43(63)(1):41--57, 2008.

\bibitem[Pet]{Pet}
D.~Peterson.
\newblock Quantum cohomology of $G/P$,.
\newblock Lecture Notes, Cambridge, MA, Spring: Massachusetts Institute of
  Technology, 1997.

\bibitem[Zhu]{Zhu96}
Yongchang Zhu.
\newblock Modular invariance of characters of vertex operator algebras.
\newblock {\em J. Amer. Math. Soc.}, 9(1):237--302, 1996.

\end{thebibliography}

\end{document}